\documentclass{amsart}
\usepackage{amssymb,amsfonts,eucal}
\usepackage[dvips]{graphics}
\usepackage[utf8]{inputenc}
\usepackage{pst-pdf}
\usepackage{pdflscape}
\usepackage{amsmath}

\newtheorem{thm}{Theorem}[section]
\newtheorem{mthm}{Main Theorem}[section]
\newtheorem{prop}{Proposition}[section]
\newtheorem{cor}[thm]{Corollary}
\newtheorem{lemma}{Lemma}[section]

\theoremstyle{definition}
\newtheorem{defn}[thm]{Definition}

\theoremstyle{definition}

\newtheorem*{rem}{Remark}

\theoremstyle{definition}

\newcommand{\real}{\mathbb{R}}

\numberwithin{equation}{section}

\begin{document}


\title[Symmetric Periodic Orbit in the Octahedral Six-Body Problem]{An Existence Proof of a Symmetric Periodic Orbit in the Octahedral Six-Body Problem}
\author{Anete S. Cavalcanti$^1$
\\}
\address{$^1$Universidade Federal Rural de Pernambuco\\
     Depto de Matemática\\
     Rua Dom Manoel de Medeiros, sn, Recife, PE\\
     52171-900, Brasil\\
     Phone: 55 81 3320-6486.\\ \textbf{CAPES' scholarship}}
\email{anete.soares@ufrpe.br;anete.soares.cavalcanti@gmail.com}     


\begin{abstract}
We present a proof of the existence of a periodic orbit for the Newtonian six-body problem with equal masses. This orbit has three double collisions each period and no multiple collisions. Our proof is based on the minimization of the Lagrangian action functional on a well chosen class of symmetric loops.

\end{abstract}
\maketitle

\section{Introduction}
The \textit{Variational Methods} applied to the n-body Newtonian problem allows to prove the existence of periodic  orbits, in most cases with some symmetry. It was exploited by the Italian school in the 90's ~\cite{Coti Zelati}, ~\cite{DGM},~\cite{SeT}. They gave new periodic solutions for a mechanical systems with potentials satisfying a hypothesis called \textit{strong force}, which excludes the Newtonian potential. The \textit{strong force hypothesis} was introduced by Poincaré, see ~\cite{Poincare}. This method was only exploited for the Newtonian potential much later. We can cite three works in this area ~\cite{CM}, ~\cite{Venturelli} and ~\cite{Shibayama}. In the first one it is proved the existence of a new periodic solution in the planar three-body problem, where all bodies are on the same curve, and this curve has a figure-eight shape; in the second one (our most prominent reference for the present paper) it is proved the existence of the collinear solution for the three-body problem already discovery numerically in $1956$ by J. Schubart ~\cite{Schubart}); in the third one, the author studies the existence of periodic solutions for subsytems where the dimension of the configuration space can be reduced to $d=2$. In this paper we proved a variational existence proof of a periodic orbit in the \emph{octahedral six-body problem} with equal masses. Next we explain the main ideas on the prove.

Let us consider six bodies with equal masses in $\real^3$. We assume that every coordinate axis contains a couple of bodies, and they are symmetric with respect to the origin, that is the center of mass of the system. This is the octahedral six-body problem. Our orbit starts with a collision in the $x-$axis and the other $4$ bodies form a square on the ortogonal plane. Let us denote by $T$ the period of the solution. During the  \textit{first sixth} of the period, the bodies on the $x-$axis move away, the two bodies on the $z-$axis also move away, while the bodies on $y-$axis approach. At time $t=T/6$ the bodies on the $z-$axis are  point and they approach each other as  $t\in [T/6,T/3]$. At time $t=T/3$ the bodies on $y-$axis are at a double collision, and the other bodies form a square on the orthogonal plane. In the second third of the period, the motion is the same as above, after having exchanged $x$ by $y$, $y$ by $z$ and $z$ by $x$. That is to say the solution satisfies the symmetry condition $x(t-T/3)=y(t)$, $y(t-T/3)=z(t)$ and $z(t-T/3)=x(t)$. See figure $1$

	\begin{figure}[h] 
\begin{center}
\scalebox{.09}{\includegraphics{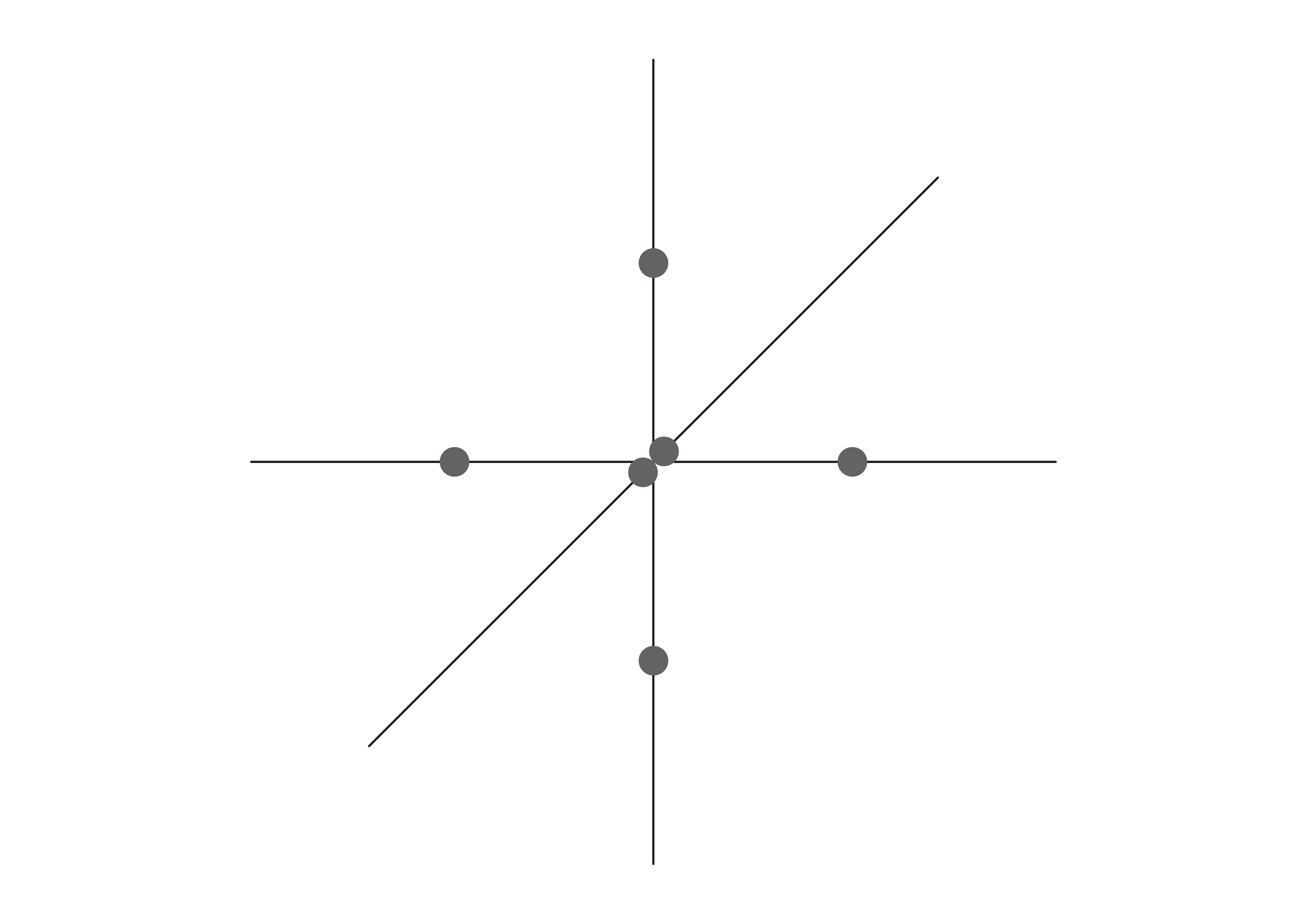}} \scalebox{.09}{\includegraphics{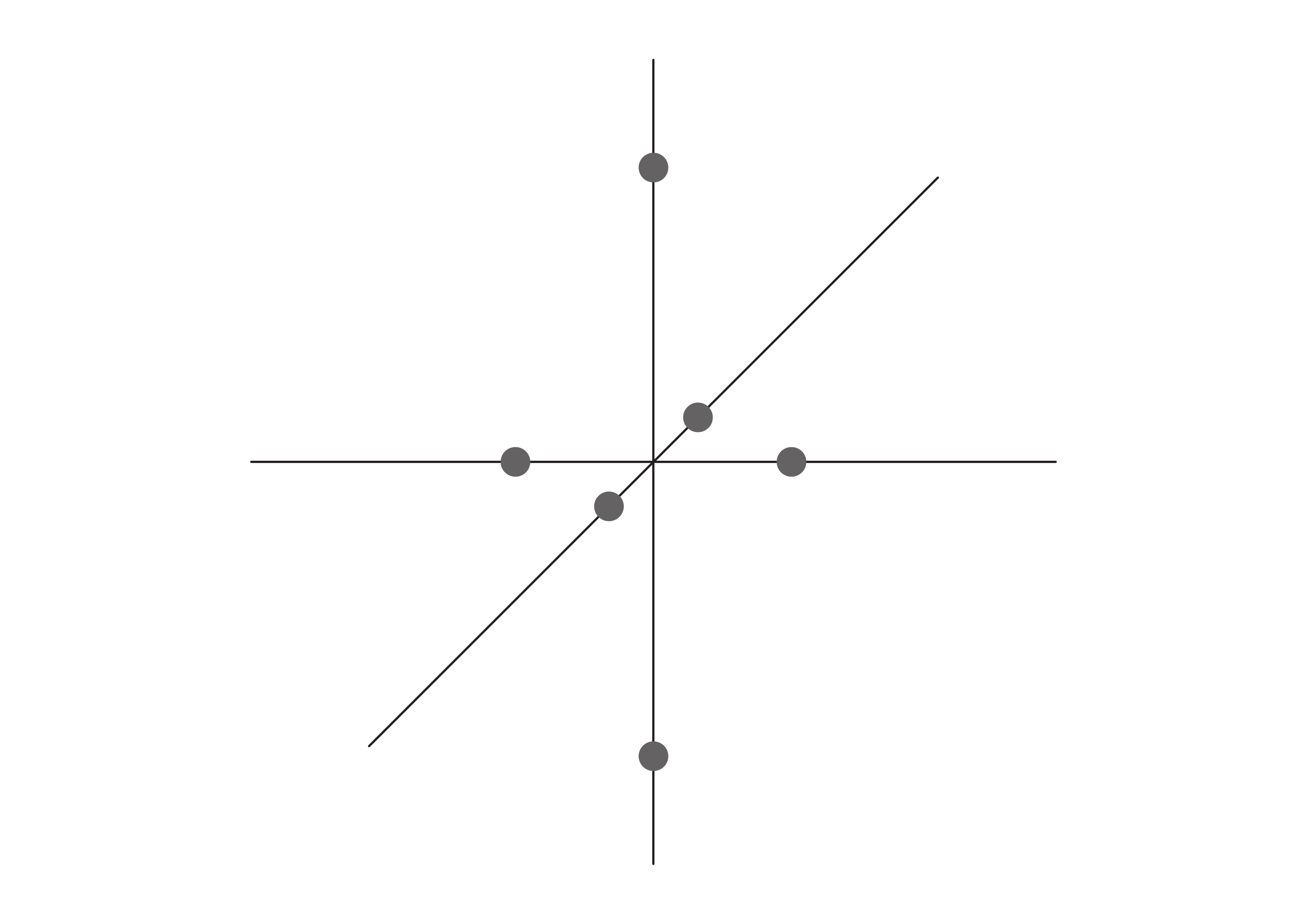}} \scalebox{.09}{\includegraphics{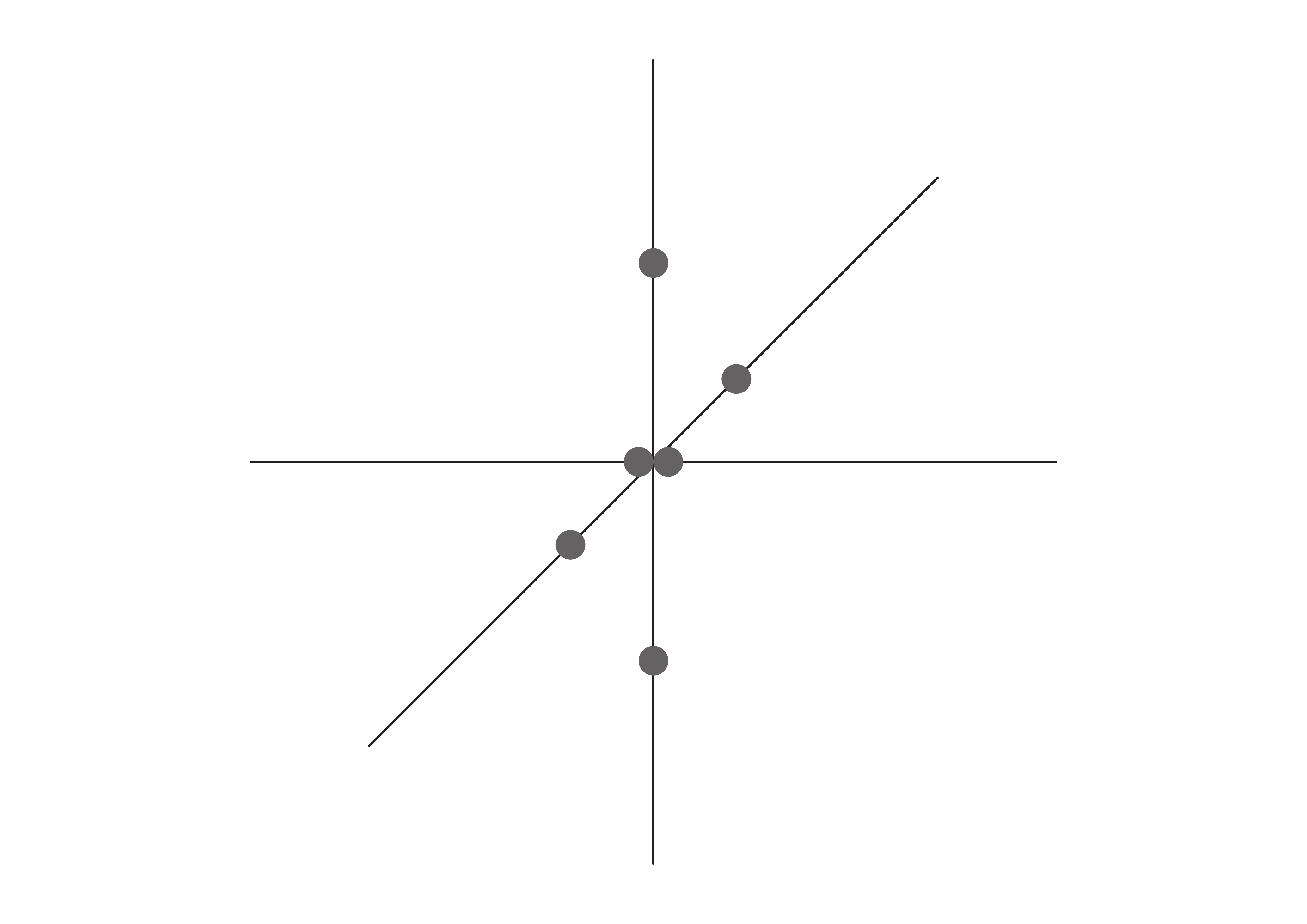}} 
\caption{A sketch of the first sixth of the orbit.}
\label{f1}
\end{center}
\end{figure}

In section 2 we will introduce the formal aspects of the problem, we will write the equations of motion and explain the variational setting. In particular, we will prove that the lagrangian action functional is \textit{coercive}. In section 3 we show that a minimizer has no other collisions besides those imposed by the choice of set of loops where we minimize the action funcional, and that all collisions are double (\emph{i.e.} there are no quadruple collision). In section 4 we regularize all possible collisions.  Our conclusions are summarized in the Main Theorem, which is proved in section 4.

\section{Equations of Motion and Variational Setting}

Notice that if at a certain instant the configuration of the bodies is octahedral (and the velocities too), the configuration is octahedral for all time. To describe the octahedral six-body problem with equal masses we will identify the configuration space with $\real^3$. Let the masses $m_1,m_2$ be on $x$-axis, at positions $q_1=(x,0,0)$ and $q_2=(-x,0,0)$, respectively . Analogously, $m_3,m_4$ are on $q_3=(0,y,0), q_4=(0,-y,0)$ and $m_5,m_6$ on $q_5=(0,0,z),$ $q_6=(0,0,-z)$. 

We can identify the configuration space $\mathcal{X}$ with $\real^3$, and every configuration with a vector $X=(x,y,z)\in\real^3=\mathcal{X}$.

Without loss of generality we can normalize the masses, $m_i=\frac{1}{2}, i=1,...,6$. So the kinetic energy and potential function assume the form:
$$K(X,\dot{X})=\frac{1}{2}(\dot{x}^2+\dot{y}^2+\dot{z}^2),$$
$$U(X)=\left(\frac{1}{\sqrt{x^2+y^2}}+\frac{1}{\sqrt{x^2+z^2}}+\frac{1}{\sqrt{y^2+z^2}}\right)+\frac{1}{8}\left(\frac{1}{x}+\frac{1}{y}+\frac{1}{z}\right).$$

A configuration is said to be a collision if at least two bodies occupy the same position in space. Due to the symmetry of the configuration, when this happens, one of the three Cartesian coordinate is equal to zero. We denote by $\Delta$ the set of collision configurations, and by $\hat{\mathcal{X}}=\mathcal{X}\setminus\Delta$ the open set of collision-free configurations. A configuration $q=(x,y,z)$ is said to be a \textit{double collision} if only one of the three Cartesian coordinates is equal to zero. It is said that to be a \textit{quadruple collision} if only two of the three Cartesian coordinates are equal to zero, and it is said to be a total collision if $x=y=z=0$. One can see that the potential becomes singular at collisions. Therefore the Hamiltonian ($H=K-U$) and the Lagrangian ($L=K+U$) beccome sigular at the collisions. We will prove in section $\ref{collisions}$ that there is no quadruple and total collision. In section 4 we will show that for an action minimizer the double collisions are regularize. The equations of motion are Euler-Lagrange equations for the lagrangean $L=K+U$. They can be written:

\begin{equation} \label{eq movi q}
\left\{ \begin{array}{c}
\ddot{x}=\frac{\partial U}{\partial x}\\
\ddot{y}=\frac{\partial U}{\partial y}\\
\ddot{z}=\frac{\partial U}{\partial z}.
\end{array}\right.
\end{equation}
Throughout the paper, $X\cdot Y$ will denote the standard scalar product in $\real^3$.

\begin{mthm}\label{main theorem}There is a $T$-periodic solution $X=(x,y,z):\real\rightarrow\real^3$ for the octahedral 6-body problem with equal masses such that: \label{Main Theorem}
\begin{itemize}
\item[(i)] $x(0)=0$, $y(0)>0$ and $z(0)>0$;
\item[(ii)] $t\mapsto x(t)$ and $t\mapsto z(t)$ are increasing functions and $t\mapsto y(t)$ is decreasing in $[0, T/6]$. While, in $[T/6,T/3]$, $t\mapsto z(t)$ becomes decreasing, $t\mapsto x(t)$ is still increasing and $t\mapsto y(t)$ is still decreasing;
\item[(iii)] $X$ verifies some symmetry conditions:
$$\left\{\begin{aligned} 
x(t-T/3)=y(t)  \\
y(t-T/3)=z(t)\\
z(t-T/3)=x(t) 
\end{aligned},
\right.\quad\quad\left\{\begin{aligned} 
x(-t)=x(t)  \\
y(-t)=z(t)\\
z(-t)=y(t) 
\end{aligned}
\right..$$
\item[(iv)] the collisions in this orbit are double collisions at $t=0,T/3,2T/3, etc.$. All collisions are regularized. 
\end{itemize}
\end{mthm}
Notice that if we assume that the orbit starts with a double collision on the $x$-axis, \textit{i.e.} $x(0)=0$, condition (iii) implies $y(T/3)=z(2T/3)=0$, $y(0)=z(0)$ and $\dot{z}(T/6)=\dot{x}(3T/6)=\dot{y}(5T/6)=0$.

Recall the Sobolev space $H^1(\real/{T\mathbb{Z}},\mathcal{X})$ is the vector space of absolutely continuous $T$-periodic loops $X:\real\rightarrow \mathcal{X}$ satisfying
$$\int_0^T\left\|\dot{X}(t)\right\|^2dt<+\infty.$$
$H^1(\real/{T\mathbb{Z}},\mathcal{X})$ is endowed with the following Hilbert product:
$$<X,Y>=\int_0^T(\dot{X}(t)\cdot\dot{Y}(t)+X(t)\cdot Y(t))dt, \quad\quad X,Y\in H^1(\real/{T\mathbb{Z}},\mathcal{X}).$$

and $\left\|\cdot\right\|_{H^1}$ denote the associated norm. The \textit{Lagrangian action functional} is defined by
$$\mathcal{A}:H^1(\real/{T\mathbb{Z}},\mathcal{X})\rightarrow \real_+\cup\{+\infty\}, \quad\quad \mathcal{A}(X)=\int_0^T L(X(t),\dot{X}(t))dt,$$
where $L(X(t),\dot{X}(t))=+\infty$ if $X$ is a collision configuration. The Lagrangian action functional is differentiable on the open set of collision-free orbits. A classical computation shows that collision-free critical points of $\mathcal{A}$ are $T$-periodic solutions of the problem (see ~\cite{Venturelli these}). The action functional is also semicontinuous with respect to the weak topology of the space $H^1(\real/{T\mathbb{Z}},\mathcal{X})$ (see ~\cite{Venturelli these}). 

Let $G$ be a group with two generators which satirfy the relations:
$$g^3=h^2=1, \quad\quad \text{and}\quad (hg)^2=1,$$
$G$ is in fact the diedral group $D_3$. We consider the action of $G$ on $H^1(\real/{T\mathbb{Z}},\mathcal{X})$ defined by
$$\left\{\begin{aligned} 
g(x,y,z)(t)=(z,x,y)(t-T/3),  \\
h(x,y,z)(t)=(x,z,y)(-t). 
\end{aligned}
\right.$$

So, it defines an action of $G$ on $H^1(\real/{T\mathbb{Z}},\mathcal{X})$. Since we assume that all masses are equal, it is easy to see that $G$ acts by isometries on $H^1(\real/{T\mathbb{Z}},\mathcal{X})$ and leaves invariant the action functional. Notice that, since we identify $\mathcal{X}$ with $\real^3$, the invariance of a loop $X\in H^1(\real/T\mathbb{Z},\mathcal{X})$ by $g$ can be interpreted by saving that $X(t+T/3)$ is obtained by a rotaion $X(t)$ of an angle $2\pi/3$ around the axis generated by the vector  $(1,1,1)$; the invariance by $h$ is equivalent to says that $X(-t)$ is symmetric to $X(t)$ with respect to the plan $y=z$. Let $\Lambda_G=\{X\in H^1(\real/{T\mathbb{Z}},\mathcal{X}): gX=X \text{ and } hX=X\}$, the subspace of invariant loops. Acoording to Palais principle (see \cite{Palais}), a critical point of $\left.\mathcal{A}\right|_{\Lambda_G}$ is still a critical point of $\mathcal{A}$, thus it is a $T$-periodic solution of the problem. Remark that since invariant loops $X\in\Lambda_G$, have $X(T/3-T)$ and $X(t)$ symmetric with respect to the plane $x=y$, they satisfy statement (iii) of the Main Theorem. In order to use the \textit{Direct Method in the Calculus of Variations}, see ~\cite{Venturelli these}, we need to restrict the action functional to a closed subset  $\Omega$ of $\Lambda_G$, such that $\left.\mathcal{A}\right|_\Omega$ is \textit{coercive}, \textit{i.e.}, the set of loops verifying $\left.\mathcal{A}\right|_{\Omega}(X)<C$ is bounded for every $C>0$. Notice that $\left.\mathcal{A}\right|_{\Lambda_G}$ is not coercive. Indeed, we can construct a sequence of constant loops, for instance $X^{(n)}(t)=(n, n, n)$, such that $X^{(n)}\in\Lambda_G$ and
$$\mathcal{A}(X^{(n)})\rightarrow 0 \quad\quad \text{as} \quad\quad n\rightarrow +\infty.$$
Therefore, since $\left.\mathcal{A}\right|_{\Lambda_G}(X)$ is everywhere strictly positive, its infimum is not achieved. Let
$$S=\{ (x,y,z)\in\mathcal{X}:x\geq 0,\quad y\geq 0,\quad z\geq 0\},$$
we will define $\Omega$ as
$$\Omega=\{X=(x(t),y(t),z(t))\in\Lambda_G: x(0)=0 \quad\quad \text{and} \quad\quad X(t)\in S\quad \forall \quad t\in\real\}.$$

\begin{figure}
\begin{center}
\scalebox{0.35}{\includegraphics{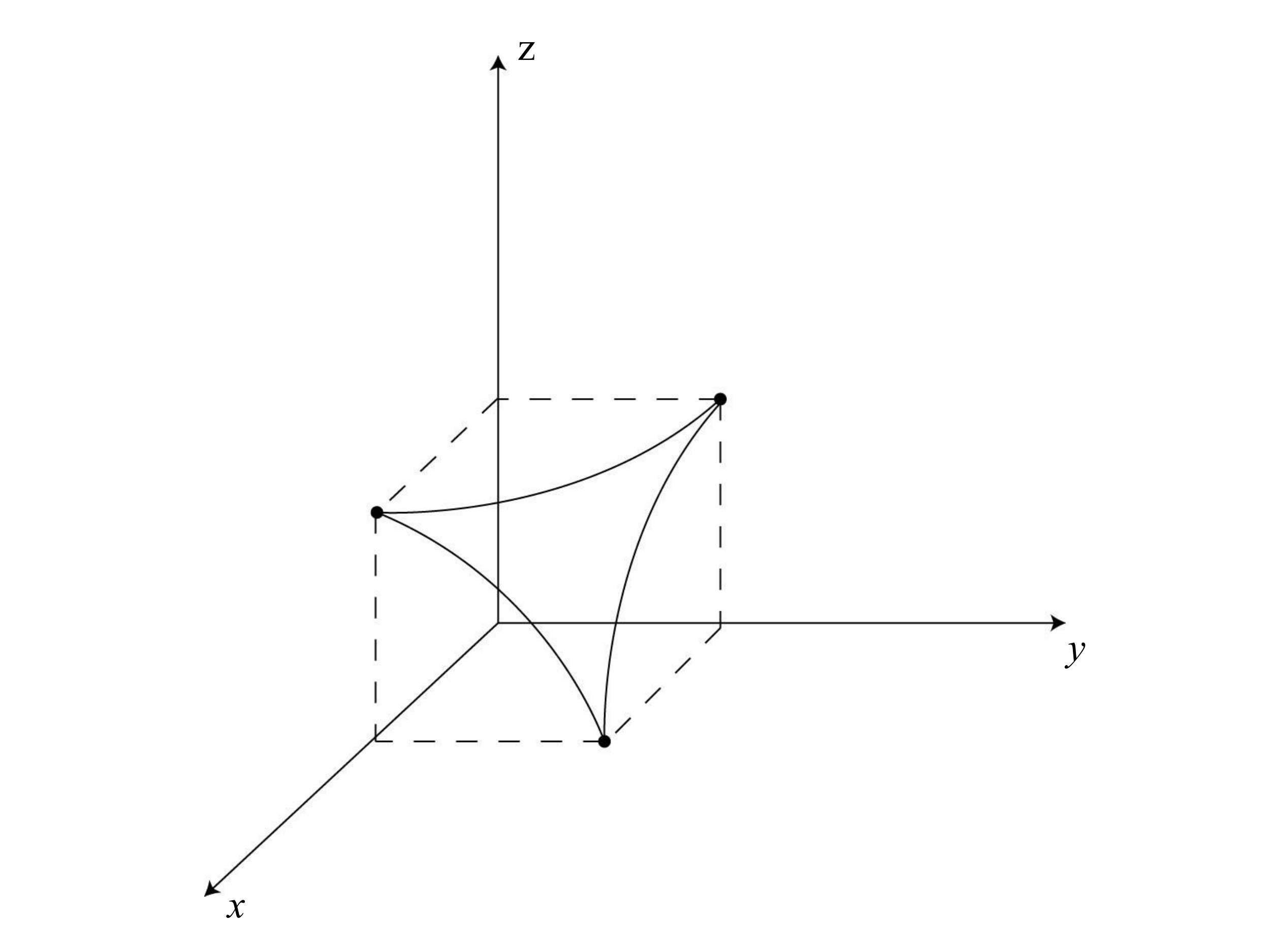}}
\caption{Orbit in the configuration space.}
\label{f2}
\end{center}
\end{figure}

\begin{prop} $\left.\mathcal{A}\right|_\Omega$ is coercive and has a minimizer.
\end{prop}

\begin{proof}
Consider $X\in\Omega$, by symmetry we will study only the first third of the orbits. Let $X(\overline{t})$ the maximum of $\left\|X(t)\right\|$ with $t\in[0,T/3]$. In order to prove the coercivity we will check that $\mathcal{L}(X)\geq a \left\|X(\overline{t})\right\|$, where $\mathcal{L}$ is the length of the curve $\left.X\right|_{[0,T/3]}$ and $a$ is a positive real number  independent of $X$. If $X(\overline{t})=0$ the inequality, $\mathcal{L}(X)\geq a \left\|X(\overline{t})\right\|$, is trivially satisfied, so we can suppose $X(\overline{t})\neq0$. Let us write $l=\sqrt{y^2(0)+z^2(0)}=\sqrt{x^2(T/3)+z^2(T/3)}$ and $d=\left\|X(\overline{t})\right\|$. Let $\alpha$ be the angle between $X(\overline{t})$ and the bissectrix of the first quadrant of the plane $x=0$; let  $\beta$ be  the angle between $X(\overline{t})$ and the bissectrix of the first quadrant of the plane $y=0$;  and let $\theta$ be the angle between $X(0)$ and $X(T/3)$, so $\theta=\pi/3$. Since $X\in\Omega$, by definition of $\Omega$ and $\Lambda_G$, we know that $X(0)$ is in the bissectrix of the first quadrant of the plane $x=0$, and $X(T/3)$ is in the bissectrix of the first quadrant of the plane $y=0$. The distance between $X(0)$ and $X(\overline{t})$ is $\sqrt{l^2+d^2-2ld\cos\alpha}$, and the distance between $X(\overline{t})$ and $X(T/3)$ is $\sqrt{l^2+d^2-2ld\cos\beta}$. Hence:

$$\begin{aligned}
\mathcal{L} & \geq\sqrt{l^2+d^2-2ld\cos\alpha}+\sqrt{l^2+d^2-2ld\cos\beta},\\
\end{aligned}
$$
Since the two terms on the right-hand side of the inequality above are decreasing with  respect to $\alpha$ and $\beta$, and $\alpha+\beta\geq\theta=\pi/3$, we can write
$$\mathcal{L}\geq \sqrt{l^2+d^2-2ld\cos\alpha}+\sqrt{l^2+d^2-2ld\cos(\pi/3-\alpha)}.$$ 
Notice that $\sqrt{l^2+d^2-2ld\cos\alpha}$ is greater than or equal to the distance from $X(\overline{t})$ to the bissectrix of first quadrant of the plane $x=0$, which is $d\sin\alpha$. In a similar way we have that 
$$\sqrt{l^2+d^2-2ld\cos\left(\frac{\pi}{3}-\alpha\right)}\geq d\sin\left(\frac{\pi}{3}-\alpha\right) .$$ So

\begin{equation} \label{comprimento}
\mathcal{L}\geq d(\sin\alpha+\sin\left(\frac{\pi}{3}-\alpha\right))\geq a||X(\overline{t})||,
\end{equation}
where 
$$a=\text{min}_{\alpha\in[0,\pi/3]}\{\sin\alpha+\sin(\pi/3-\alpha)\}.$$
Given $C>0$, let $\Omega_C=\{X\in\Omega : \mathcal{A}(X)\leq C\}$. Let us prove that this set is bounded, this will imply that $\left.\mathcal{A}\right|_{\Omega}$ is coercive. By the Cauchy-Schwartz inequality, $(\ref{comprimento})$ and the invariance of $\mathcal{A}$ by the action of $G$, we have:
$$\begin{array}{ccc}
\int_0^TK(X,\dot{X})dt & \geq\frac{9}{2T}\left(\int_0^{\frac{T}{3}}\left\|\dot{X}(t)\right\|dt\right)^2 & \geq\frac{9}{2T} a^2\text{max}_{t\in[0,T/3]}\left\|X(t)\right\|^2\\

 & \geq\frac{9}{2T^2} a^2\int_0^{T}\left\|X(t)\right\|^2dt . & 
\end{array}$$
On the other hand, if $X\in\Omega_C$ $$\int_0^T\left\|\dot{X}(t)\right\|^2dt=\int_0^T2K(X,\dot{X})dt\leq2\mathcal{A}(X)\leq 2C,$$
so  $$\left\|X\right\|_{H^1}^2=\int_0^T\left(\left\|X(t)\right\|^2+\left\|\dot{X}(t)\right\|^2\right)dt\leq 2C\left(\frac{2T^2}{9a^2}+1\right).$$ 

In order to prove that $\mathcal{A}|_\Omega$ has a minimizer, let $(c_n)_{n\in\mathbb{N}}$ be a decreasing sequence of real positive numbers converging to inf$\mathcal{A}|_\Omega$. By the last inequality and the weak lower semicontinuity of $\mathcal{A}_{\Omega}$ every set
$$\Omega_n=\{X\in\Omega,\quad \mathcal{A}(X)\leq c_n\}$$
is bounded in the norm of $H^1$ and closed in the weak topology. Since $H^1(\real/T\mathbb{Z},\mathcal{X})$ is a separable Hilbert space, $\Omega_n$ is compact in the weak topology ~\cite{Brezis}. Since the sequence $(c_n)_{n\in\mathbb{N}}$ is decreasing we have $\Omega_{n+1}\subset\Omega_n$, therefore $$\cap_{n\in\mathbb{N}}\Omega_n\neq\emptyset.$$  Notice that any element of $\cap_{n\in\mathbb{N}}\Omega_n$ is a minimizer. 

\end{proof}

\section{Elimination of Extra Collisions}
\label{collisions}

In this section we will study collisions. We will prove that for a minimizer of $\mathcal{A}|_\Omega$, each segment of orbit without collisions are solutions of the octahedral problem, hence it is indeed a solution of the Newtonian six-body problem. We shall prove that there are no \emph{other} collisions, like quadruple or total collision, and no extra double collision. In the next section we will show that the double collisions (in $t=0,T/3,2T/3 \quad \text{mod}\text{  }T$ ) are \textit{regularizable}.

\begin{prop} \label{without collisions}If $X$ is a minimizer of $\left.\mathcal{A}\right|_\Omega$, it is a solution of the $6-$body problem in every interval without collisions. \end{prop}
\begin{proof}

Denote by $T_c(X)\subset\real$ the set of collision times of $X$. Since $X$ is a continuous path with finite action, $T_c(X)$ is closed and has zero measure. $\real\setminus T_c(X)$ is the union of open intervals; let $(a,b)$ be one of those intervals. Since in our orbit $t=0,T/3$ are collision times, $b-a\leq T/3$. By invariance of $X$ under the action of $G$, we can assume that $(a,b)\subset(0,T/3)$. Let
$$H_0^1([a,b],\mathcal{X})=\{\xi\in H^1([a,b],\mathcal{X}):\quad \text{supp}(\xi)\subset(a,b)\}.$$
Let us define $\xi$ on $\real$ by setting $\xi(t)=0$ for $t\in[0,T] \setminus(a,b)$, and consider the unique $T$-periodic constructed from $\xi$, still denoted $\xi$. In order to symmetrize $\xi$ we will define $\xi_G$:
$$\xi_G=\frac{1}{|G|}\sum_{g\in G}{g\xi},$$
where $|G|=6$ is the cardinality of the group $G$. It is clear that $\xi_G$ and supp$(\xi_G)$ are containd in the complement of $T_C(X)$, hence $s\mapsto\mathcal{A}(X+s\xi_G)$ is differentiable for $s$ small enough. Moreover:

$$\mathcal{A}(X+s\xi_G)=\int_0^T\left(\frac{\left\|\dot{X}(t)\right\|^2}{2}+s\dot{X}(T)\cdot\dot{\xi}_G(t)+\frac{s^2}{2}\left\|\dot{\xi_G(t)}\right\|^2+U(X(t)+s\xi_G(t))\right)dt, $$ hence $$\left.\frac{d}{ds}\mathcal{A}(X+s\xi_G)\right|_{s=0}=\int_0^T(\dot{X}(t)\cdot\dot\xi_G(t)+\nabla U(X(t))\cdot\xi_G(t))dt.$$
Since $X$ is a minimizer of $\mathcal{A}$, in every direction where $\mathcal{A}$ is differentiable at $X$, the derivate is equal to zero, hence
\begin{equation} \label{sei la}
\begin{array}{ccc}
\left.\frac{d}{ds}\mathcal{A}(X+s\xi_G)\right|_{s=0}&=&\int_0^T(\dot{X}(t)\cdot\dot\xi_G(t)+\nabla U(X(t))\cdot\xi_G(t))dt\\
 &=&\int_0^T(\dot{X}(t)\cdot\dot\xi(t)+\nabla U(X(t))\cdot\xi(t))dt\\
&=&\int_a^b(\dot{X}(t)\cdot\dot\xi(t)+\nabla U(X(t))\cdot\xi(t))dt=0\\
\end{array}
\end{equation}

A classical result of Calculus of Variations (see ~\cite{Young}) shows that $(\ref{sei la})$ implies:
$$\ddot{X}=\nabla U(X),\quad t\in(a,b)$$
so, $\left.X\right|_{(a,b)}$ is a solution of the six-body problem. This finishes the proof.
\end{proof}

Next, we will prove that the only \textit{central configuration} in the newtonian symmetric octahedral six-body problem is the \textit{octahedron}. A configuration in the $n$-body problem with positions $({q_1},...,{q_n})$ and masses $m_1,...m_n$ is called \textit{central} if $$\sum_{{j=1,...,n}\atop{j\neq i}}m_j\frac{q_i-q_j}{|q_i-q_j|^3}+\lambda{q_i}=0\quad\quad \text{for}\quad\quad i=1,...,n.$$ 
In our problem $n=6$, $q_1=(x,0,0)$, $q_2=(-x,0,0)$, $q_3=(0,y,0)$, $q_4=(0,-y,0)$, $q_5=(0,0,z)$, $q_6=(0,0,-z)$ and $m_i=1/2$. So the previous condition is equivalent to the existence of $\lambda\in\real$ such that:  
\begin{equation} \label{sistema c.c.}
\left\{\begin{array}{ccc}
\frac{1}{8x^3}+\frac{1}{(x^2+y^2)^{3/2}}+\frac{1}{(x^2+z^2)^{3/2}}& = &-\lambda \\
\frac{1}{8y^3}+\frac{1}{(x^2+y^2)^{3/2}}+\frac{1}{(y^2+z^2)^{3/2}}& = &-\lambda \\
\frac{1}{8z^3}+\frac{1}{(x^2+z^2)^{3/2}}+\frac{1}{(y^2+z^2)^{3/2}}& = &-\lambda 
\end{array}.\right. 
\end{equation}

\begin{lemma}\label{unique c.c.} The unique central configuration in the octahedral six-body problem is the regular octahedron.
\end{lemma}
\begin{proof} To simplify the proof, we will ommit dependence on $t$. 
It is easy to check that if $x=y=z=a>0$, then $X=(x,y,z)=(c,c,c)$ is a solution of the system  (\ref{sistema c.c.}) with $\lambda=-\frac{1+2^{5/2}}{8c^3}$, hence it is a central configuration. On the other hand, suppose that $X=(x,y,z)$ is a solution of (\ref{sistema c.c.}) and, by contradiction, that the condition $x=y=z$ is not satisfied. Without loss of generality, we can assume that $x>y$. We have
$$\frac{1}{8x^3}<\frac{1}{8y^3}\quad\quad \text{and} \quad\quad \frac{1}{(x^2+z^2)^{3/2}}<\frac{1}{(y^2+z^2)^{3/2}},$$
and we obtain a contradiction with the two first equations of (\ref{sistema c.c.}).

\end{proof}

\begin{prop} \label{total collision} A minimizer of $\left.\mathcal{A}\right|_\Omega$ is free of total collision.
\end{prop}
\begin{proof}

The invariance of $\mathcal{A}$ by the action of  $G$ implies that we just need to prove that if $X$ is a minimizer, $\left.X\right|_{[0,T/6]}$ is free of total collision. Following ~\cite{Venturelli}, we introduce the notation
$$\mathcal{A}_{\frac{1}{6}}:H^1([0,T/6],\mathcal{X})\rightarrow \real_+\cup\{+\infty\}, \quad\mathcal{A}_{\frac{1}{6}}(X)=\int_0^{\frac{T}{6}}\left(K(X,\dot{X})+U(X)\right)(t)dt;$$
$$\Omega_{\frac{1}{6}}=\{X\in H^1([0,T/6],S): \quad x(0)=0,\quad y(0)=z(0), \quad x(T/6)=y(T/6)\}.$$
Since $\mathcal{A}$ is invariant by the action of $G$, $X$ is a minimizer of $\mathcal{A}$ if and only if $X|_{[0,T/6]}$ is a minimizer of $\mathcal{A}_{\frac{1}{6}}$.

By contradiction, suppose that there is a total collision at the instant $\overline{t}\in[0,T/6]$. First of all we will prove that $X|_{[0,T/6]}$ is one half of the homothetic ejection-collision orbit with central configuration the regular octahedron. Afterwards we explicit a deformation of $X$, $\overline{X}_\varepsilon$, which decreases the action, hence producing a contradiction. 

We denote by $X_{c^+}$ the normalized central configuration in our problem: $\left\|X_{c^+}\right\|^2=1$. Let us term $r(t)=\left\|X(t)\right\|$ and $\tilde{U}(X)=\left\|X\right\|U(X)$. Since $X_{c^+}$ is the unique central configuration for the octahedral six-body problem, $\mathcal{G}=\tilde{U}(X_{c^+})$ is a minimum of $\tilde{U}$ on $S$. Then
$$\mathcal{A}_{\frac{1}{6}}\left(\left. X\right|_{[0,\frac{T}{6}]}\right)\geq\int_0^{\frac{T}{6}}\left(\frac{\dot{r}^2}{2}+\frac{\mathcal{G}}{r}\right)(t)dt=\mathcal{A}_{\frac{1}{6}}(rX_{c^+}).$$
Let us consider now the function $$a(s)=\mathcal{A}_\frac{1}{6}(sX_{c^+})=\int_0^{T/6}\left(\frac{\dot{s}(t)^2}{2}+\frac{\mathcal{G}}{s(t)}\right)dt,$$
defined on the set of one dimensional paths $s\in H^1([0,T/6],\real_+)$ satisfying $s(\overline{t})=0$. The minimum of $a$ is achived by the path $t\mapsto v(t)$ obtained by joining one-half of a collision-ejection solution of period $2\overline{t}$, with one-half of an ejection-collision solution  of period $2(T/6-\overline{t})$, for the one-dimensional Kepler problem: 
\begin{equation}
\ddot{v}=-\frac{\mathcal{G}}{v^2},
\end{equation}

We remark that the path $t\mapsto v(t)X_{c^+}$ is not necessarly an element of $\Omega_{\frac{1}{6}}$, but we still have 
$$\mathcal{A}_{\frac{1}{6}}\left(\left. X\right|_{\left[0,\frac{T}{6}\right]}\right)\geq \mathcal{A}_{\frac{1}{6}}\left(vX_{c^+}\right),$$
By the Virial Theorem, the action of a homothetic collision-ejection motion of period $\tau$, denoted here $a(\tau)$, is equal to the action of a circular motion of same period, $\tau$, thus
$$a(\tau)=\alpha_0\tau^{\frac{1}{3}}, \quad\quad \alpha_0=\frac{3}{2^{\frac{1}{3}}}(\pi\mathcal{G})^\frac{2}{3}.$$
Using an argument introduced by Gordon ~\cite{Gordon}, essentially the concavity of $\tau\mapsto \alpha_0\tau^\frac{1}{3}$, we have: 
$$\mathcal{A}_\frac{1}{6}(vX_{c^+})=\frac{\alpha_0}{2^\frac{2}{3}}\left(\overline{t}^\frac{1}{3}+(T/6-\overline{t})^\frac{1}{3}\right)\geq\frac{\alpha_0}{2^\frac{2}{3}}(T/6)^\frac{1}{3}=\mathcal{A}_\frac{1}{6}{(\overline{X})},$$
where
$$\overline{X}:[0,T/6]\rightarrow S, \quad \overline{X}(t)=u(t)X_{c^+}$$
is one-half of the homothetic ejection-collision solution of period $2T/6$. Notice that $\overline{X}$ is now an element of $\Omega_{\frac{1}{6}}$. Now we will  construct a deformation of $\overline{X}=(\overline{x},\overline{y}_,\overline{z})$ decreasing the action. Let
$$f_\varepsilon(t)=\left\{\begin{array}{cc}
1,& t\in[0,\varepsilon^\frac{3}{2}]\\
\frac{\varepsilon+\varepsilon^\frac{3}{2}-t}{\varepsilon},& t\in[\varepsilon^\frac{3}{2},\varepsilon^\frac{3}{2}+\varepsilon]\\
0,& t\in [\varepsilon^\frac{3}{2}+\varepsilon, T/6]
\end{array}.\right.$$
Let us define $\overline{X}_\varepsilon=\overline{X}+\varepsilon f_\varepsilon(t)Z$, where $Z=(z_1,z_2,z_3)$ is a fixed configuration satisfying $z_2=z_3>0$ and $z_1=0$. We can say, for instance, $Z=(0,1,1)$. We remark that $\overline{X}_\varepsilon$ is still in $\Omega_\frac{1}{6}$, we can write the diference between the action of $\overline{X}_\varepsilon$ and the action of $\overline{X}$ as the sum of three terms:
$$\Delta\mathcal{A}_\frac{1}{6}=\mathcal{A}_\frac{1}{6}(\overline{X}_\varepsilon)-\mathcal{A}_\frac{1}{6}(\overline{X})=A_1(\varepsilon)+A_2(\varepsilon)+A_3(\varepsilon),$$
where 
$$A_1(\varepsilon)=\int_0^{\varepsilon^\frac{3}{2}}\left(U(\overline{X}_\varepsilon)-U(\overline{X})\right)(t)dt,\quad A_2(\varepsilon)=\int_{\varepsilon^\frac{3}{2}}^{\varepsilon^\frac{3}{2}+\varepsilon}\left(U(\overline{X}_\varepsilon)-U(\overline{X})\right)(t)dt,$$
$$A_3(\varepsilon)=\int_{\varepsilon^\frac{3}{2}}^{\varepsilon^\frac{3}{2}+\varepsilon}\left(K(\overline{X}_\varepsilon,\dot{\overline{X}}_\varepsilon)-K(\overline{X},\dot{\overline{X}})\right)(t)dt.$$
Using the classical Sundman's estimation we have
\begin{equation} \label{Sundman}u(t)=u_0t^\frac{2}{3}+O(t),\quad \dot{u}(t)=\frac{2}{3}u_0t^{-\frac{1}{3}}+O(1).\end{equation}
We can write $X_{c^+}=\frac{1}{\sqrt{3}}(1,1,1)$. Notice that each $A_i, i=1,2$ is non-positive. Let us prove that $A_1(\varepsilon)<0$. Let $\overline{A}_1(\varepsilon)$ be defined as

$$\overline{A}_1(\varepsilon)=\frac{1}{8}\int_0^{\varepsilon^\frac{3}{2}}\left(\frac{1}{\overline{y}(t)+\varepsilon}-\frac{1}{\overline{y}(t)}\right)dt,$$
Notice that $A_1(\varepsilon)\leq\overline{A}_1(\varepsilon)\leq0$.
$$\overline{A}_1(\varepsilon)=\frac{1}{8}\int_0^{\varepsilon^\frac{3}{2}}\left(\frac{1}{\frac{1}{\sqrt{3}}u_0t^\frac{2}{3}+O(t)+\varepsilon}-\frac{1}{\frac{1}{\sqrt{3}}u_0t^\frac{2}{3}+O(t)}\right)dt,$$
introducing the variable $\tau$ define by $t^\frac{2}{3}=\varepsilon\tau$, we get
$$\begin{aligned}
\overline{A}_1(\varepsilon) & =\frac{3}{16}\varepsilon^\frac{3}{2}\sqrt{3}\int_0^1\left(\frac{1}{u_0\varepsilon\tau+\varepsilon\sqrt{3}+O((\varepsilon\tau)^\frac{3}{2})}-\frac{1}{u_0\varepsilon\tau+O((\varepsilon\tau)^\frac{3}{2})}\right)\tau^\frac{1}{2}d\tau,\\
 & =\frac{3}{16}\varepsilon^\frac{1}{2}\sqrt{3}\int_0^1\left(\frac{1}{u_0\tau+\sqrt{3}+O(\varepsilon^\frac{1}{2}\tau^\frac{3}{2})}-\frac{1}{u_0\tau+O(\varepsilon^\frac{1}{2}\tau^\frac{3}{2})}\right)\tau^\frac{1}{2}d\tau,\\
 & =\frac{3}{16}\varepsilon^\frac{1}{2}\sqrt{3}\int_0^1\left[\left(\frac{1}{u_0\tau+\sqrt{3}}-\frac{1}{u_0\tau}\right)\tau^\frac{1}{2}+O(\varepsilon^\frac{1}{2})\right]d\tau,\\
 & =\frac{3}{16}\varepsilon^\frac{1}{2}\sqrt{3}\int_0^1\left(\frac{1}{u_0\tau+\sqrt{3}}-\frac{1}{u_0\tau}\right)\tau^\frac{1}{2}d\tau+O(\varepsilon).
\end{aligned}
$$
We know that $\frac{3\sqrt{3}}{16}\int_0^1\left(\frac{1}{u_0\tau}-\frac{1}{u_0\tau+\sqrt{3}}\right)\tau^\frac{1}{2}d\tau=C>0$, so
\begin{equation} \label{A1}
\overline{A}_1(\varepsilon)=-C\varepsilon^\frac{1}{2}+O(\varepsilon), \quad\quad \varepsilon\rightarrow 0^+.
\end{equation}
Now, we will prove that $A_3(\varepsilon)=O(\varepsilon)$, and with $(\ref{A1})$ we have
$$\Delta\mathcal{A}_\frac{1}{6}=A_1(\varepsilon)+A_2(\varepsilon)+A_3(\varepsilon)\leq -C\varepsilon^\frac{1}{2}+O(\varepsilon),$$
so $\Delta\mathcal{A}_\frac{1}{6}<0$ for $\varepsilon>0$ sufficiently small. And we have a contradiction, which proves our proposition.

Let us study $A_3(\varepsilon)$,

 \begin{align} A_3(\varepsilon)& =\quad\frac{1}{2}\int_{\varepsilon^\frac{3}{2}}^{\varepsilon^\frac{3}{2}+\varepsilon}\left(\left\|\dot{\overline{X}}_\varepsilon\right\|^2-\left\|\dot{\overline{X}}\right\|^2\right)dt\nonumber \\
 & \leq\frac{1}{2}\int_{\varepsilon^\frac{3}{2}}^{\varepsilon^\frac{3}{2}+\varepsilon}\left(\left\|\dot{\overline{X}}\right\|^2-2\dot{\overline{X}}\cdot Z+||Z||^2-\left\|\dot{\overline{X}}\right\|^2\right)dt\nonumber\\
&\leq\frac{1}{2}\int_{\varepsilon^\frac{3}{2}}^{\varepsilon^\frac{3}{2}+\varepsilon}\left(||Z||^2-2\dot{u}(t)X_{C_+}\cdot Z\right)dt\nonumber\\
&\leq \frac{1}{2}\int_{\varepsilon^\frac{3}{2}}^{\varepsilon^\frac{3}{2}+\varepsilon}||Z||^2=O(\varepsilon),\nonumber 
\end{align}

where the last inequality holds because $X_{c_+}\cdot Z$ is postive. This ends our proof.
\end{proof}

In order to prove that there are no quadruple collisions, we need to prove that there is not an \textit{extra} double collisions. A double collision is called \textit{extra} when it is a double collision with masses on the same axis at an instant different from $t=kT/3$, where $k\in\mathbb{Z}$.

Without loss of generality, we just have to prove that there is no extra double collisions for the cluster $\{1,2\}$ (masses $m_1,m_2$) in the interval $(0,T/6)$, and the result could be extended for all real time. Extra double collisions for the cluster $\{1,2\}$ means that $x=0$, $y>0$ and $z>0$. The other cases ($y=0,x>0,z>0$ and $z=0,x>0,y>0$) are analogous. Let us define the lagrangean and energy of that cluster in the following way:
$$L_x(X,\dot{X})=\frac{1}{2}\dot{x}^2+\frac{1}{8x} \quad \text{and} \quad h_x(X,\dot{X})=\frac{1}{2}\dot{x}^2-\frac{1}{8x},$$

following ~\cite{Venturelli}, we introduce this notation for the cluster. Then, the kinetic and potential energies are define as $$K_x(X,\dot{X})=\frac{1}{2}\dot{x}^2 \quad \text{and}\quad U_x(X,\dot{X})=\frac{1}{8x}.$$ 
\begin{defn} Let $f(t,\varepsilon)$ be a function of two real values and $g(\varepsilon)$ a function of $\varepsilon$. We say that $f(t,\varepsilon)=O_\varepsilon(g(\varepsilon)), \quad\varepsilon\rightarrow\varepsilon_0$ if $f(t,\varepsilon)=O(g(\varepsilon)),\quad \varepsilon\rightarrow\varepsilon_0$ uniformly in $t$. Analogously, we define $f(t,\varepsilon)=o(g(\varepsilon)),\quad \varepsilon\rightarrow\varepsilon_0$.
\end{defn}

\begin{prop} \label{extra double collision} Let $X=(x,y,z)$ be a minimizer of $\left.\mathcal{A}\right|_\Omega$ and $\overline{t}\in(0,T/6]$ an extra double collision of $m_1,m_2$, \textit{i.e.} $x(\overline{t})=0$, $y(\overline{t})>0$ and $z(\overline{t})>0$. There is an interval $[a,b]$ such that 
\begin{equation}\label{sol}
\overline{t}\in[a,b]\subset(0,T/6],
\end{equation}
and $t\mapsto h_x(X(t),\dot{X}(t))$ is an absolutely continuous function in the interval $[a,b]$.
\end{prop}
\begin{proof}
Since $X$ is a minimizer of $\mathcal{A}$, $\left.X\right|_{[0,T/6]}$ is a minimizer of $\left.\mathcal{A}\right|_{\Omega_\frac{1}{6}}$ and the function $t\mapsto h_x(t)=h_x(X(t),\dot{X}(t))$ is well defined in the set of non-collisions times. This is an open set with full measure. By continuity of $t\mapsto X(t)$, there exists an interval $[a,b]$ as in $(\ref{sol})$ such that only double collisions on the $x$-axis occur in the interval $[a,b]$. Let $\delta:[0,T/6]\rightarrow \real$ be a function de class $\mathcal{C}^1$ whose support is contained in $[a,b]$. The fuction
$$t_\varepsilon:[0, T/6]\rightarrow [0, T/6],\quad \tau \mapsto  t_\varepsilon(\tau)=\tau+\varepsilon\delta(\tau),$$
is monotonic and bijective if $\varepsilon$ is sufficiently small. Let us denote by $t\mapsto \tau_\varepsilon(t)$  the inverse of the map $\tau\mapsto t_\varepsilon(t)$, hence
\begin{equation}\label{pre :)}\frac{d\tau_\varepsilon}{dt}=\frac{1}{1+\varepsilon\frac{d\delta}{d\tau}(\tau_\varepsilon(t))}=1-\varepsilon\frac{d\delta}{d\tau}(\tau_\varepsilon(t))+O_\varepsilon(\varepsilon^2).
\end{equation}
By definition of $t_\varepsilon$, we have $\tau_\varepsilon(t)-t=O_\varepsilon(\varepsilon)$. Replacing in $(\ref{pre :)})$ we have
\begin{equation}\label{:)}
\frac{d\tau_\varepsilon}{dt}(t)=1-\varepsilon\frac{d\delta}{dt}(t)+O_\varepsilon(\varepsilon).
\end{equation}

The path $X_\varepsilon(\tau)=(x(t_\varepsilon(\tau)), y(\tau), z(\tau))$ is a variation of $\left.X\right|_{[0,T/6]}$, and it is an element of $\Omega_\frac{1}{6}$. We define:
$$\mathcal{A}_\frac{1}{6}^x(X_\varepsilon)=\int_0^\frac{T}{6}\left(K_x(\dot{X}_\varepsilon(t_\varepsilon(\tau)))+U_x(X_\varepsilon(t_\varepsilon(\tau)))\right)d\tau=\int_0^\frac{T}{6}\left(\frac{1}{2}\dot{x}(t_\varepsilon(\tau))^2+\frac{1}{8x(t_\varepsilon(\tau))}\right)d\tau;$$
$$\mathcal{A}_\frac{1}{6}^0(X_\varepsilon)=\int_0^\frac{T}{6}\left(K_0(\dot{X}_\varepsilon(t_\varepsilon(\tau)))+U_0(X_\varepsilon(t_\varepsilon(\tau)))\right)d\tau,$$
where \label{U_0 and K_0}
$$K_0=\frac{1}{2}(\dot{y}^2+\dot{z}^2)$$
and
$$U_0=\left(\frac{1}{\sqrt{x^2+y^2}}+\frac{1}{\sqrt{x^2+z^2}}+\frac{1}{\sqrt{y^2+z^2}}\right)+\frac{1}{8}\left(\frac{1}{y}+\frac{1}{z}\right).$$
When we make the change of variables $\tau=\tau_\varepsilon(t)$ in $\mathcal{A}_\frac{1}{6}^x(X_\varepsilon)$, we have
$$\mathcal{A}_\frac{1}{6}^x(X_\varepsilon)=\int_0^\frac{T}{6}\left(\frac{1}{2}\frac{\dot{x}(t)^2}{\frac{d\tau_\varepsilon}{dt}}+\frac{1}{8x(t)}\frac{d\tau_\varepsilon}{dt}\right)dt.$$
Since $\left.X\right|_{[0,T/6]}$ is a minimizer of $\left.\mathcal{A}_\frac{1}{6}\right|_\Omega$ and $\varepsilon\rightarrow \mathcal{A}_\frac{1}{6}(X_\varepsilon)$ is a differentiable function, using $(\ref{:)})$ we have
$$\left.\frac{d}{d\varepsilon}\right|_{\varepsilon=0}\mathcal{A}_\frac{1}{6}(X_\varepsilon)=\int_a^b\left(h_x\dot{\delta}+\frac{\partial U_0}{\partial x}\dot{x}\delta\right)dt=0\quad\quad \forall \delta.$$
Like in proposition $\ref{without collisions}$, we will use the Fundamental Lemma of the Calculus of Variations (see ~\cite{Young}), to conclude that $t\mapsto h_x(t)$ is absolutely continuous in $[a,b]$ and 
$$\dot{h}_x(t)=\frac{\partial U_0}{\partial x}(X(t))\dot{x},$$
for almost every $t\in[a,b]$.
\end{proof}

\begin{cor}Extra double collision times of a minimizer of $\left.\mathcal{A}\right|_\Omega$ are isolated in the set of collision times.
\end{cor}
\begin{proof}
Let $X$ a minimizer of $\left.\mathcal{A}\right|_\Omega$, $\overline{t}\in(0,T/6]$ an extra double collision  and $T_c(X)=\{t^*: x(t^*)y(t^*)z(t^*)=0\}$ the set of collision times. By contradiction, suppose that $\overline{t}$ is not isolated in $T_c(X)$. Without loss of generality, we can suppose that $x(\overline{t})=0$,  $y(\overline{t})>0$ and $z(\overline{t})>0$. By continuity of $X$, only collisions of $m_1,m_2$ can accumulate in $\overline{t}$. We can choose an interval $[a,b]$ such that $\overline{t}\in(a,b)$ and for all $t\in[a,b]$ there is no double collision between $m_1$ or $m_2$ and any other masses. $T_c$ is closed with zero measure, so there exists a sequence of compact intervals $([a_n,b_n])_{n\in\mathbb{N}}$ such that $a_n,b_n\rightarrow \overline{t}$ and

$$x(a_n)=x(b_n)=0 \quad\quad \text{and}\quad\quad x(t)> 0\quad\forall t\in(a_n,b_n).$$

The moment of inertia of the cluster could be written as $I_x(t)=I_x(X)(t)=\left\|x(t)\right\|^2$. Therefore, $I_x(a_n)=I_x(b_n)=0$ and $I_x(t)\neq0$ $\forall t\in(a_n,b_n)$. By proposition $\ref{without collisions}$, $X$ is a solution of octahedral symmetric six-body problem in $(a_n,b_n)$ and $I_x(X)(t)$ is $\mathcal{C}^2$ in $(a_n,b_n)$. If $s_n\in(a_n,b_n)$ denotes the maximum of $I_x$ in $[a_n,b_n]$,  $\ddot{I}_x(s_n)\leq0$. The  equation of motion tells us
$$\ddot{x}=\frac{\partial U}{\partial x}=\frac{\partial U_x}{\partial x}+\frac{\partial U_0}{\partial x},$$
and we have the generalized Lagrange-Jacobi relation
$$\begin{array}{cc}
\ddot{I}_x(t) & =4K_x+2x\left(\frac{\partial U_x}{\partial x}+\frac{\partial U_0}{\partial x}\right)\\
 & =4h_x(t)+2U_x+2x\frac{\partial U_0}{\partial x}.
\end{array}$$
But, by the last proposition $h_x$ is an absolutely continuous function in $[a,b]$, so $h_x$ and $x\frac{\partial U_0}{\partial x}$ are bounded and $U_x(t)\stackrel{t\rightarrow\overline{t}}{\longrightarrow}+\infty$. Then $\ddot{I}_x(s_n)\stackrel{n\rightarrow+\infty}{\longrightarrow}+\infty$, which is a contradiction. The corollary is proved. 
\end{proof}

Using proposition $\ref{extra double collision}$ and this corollary we will prove the following theorem.

\begin{thm} \label{thm extra double collision} A minimizer of $\left.\mathcal{A}\right|_\Omega$ does not have extra double collisions.
\end{thm}

\begin{proof}
We will make a small deformation of the hypothetical minimizer with a collision decreasing the action, similar to that we did in proposition $\ref{total collision}$. Suppose that there is a minimizer of $\left.\mathcal{A}\right|_\Omega$ with an extra double collision in the instant $\overline{t}\in(0,T/6)$ between the masses on the $x$-axis. By Sundman's estimates we have
$$x(t)=x_0|t-\overline{t}|^\frac{2}{3}+O(|t-\overline{t}|), \quad \text{and}\quad \dot{x}(t)=\frac{2}{3}x_0|t-\overline{t}|^{-\frac{1}{3}}+O(1),$$
as $t\rightarrow\overline{t}$, where $x_0>0$.  
For $\varepsilon>0$, define $f_\varepsilon:[0,T/6]\rightarrow\real_+$, by
$$f_\varepsilon(t)=\left\{\begin{array}{cc}
\frac{\varepsilon+\varepsilon^\frac{3}{2}+(t-\overline{t})}{\varepsilon},& t\in[\overline{t}-(\varepsilon^\frac{3}{2}+\varepsilon),\overline{t}-\varepsilon^\frac{3}{2}]\\
1,& t\in[\overline{t}-\varepsilon^\frac{3}{2},\overline{t}+\varepsilon^\frac{3}{2}]\\
\frac{\varepsilon+\varepsilon^\frac{3}{2}+(\overline{t}-t)}{\varepsilon},& t\in[\overline{t}+\varepsilon^\frac{3}{2},\overline{t}+\varepsilon^\frac{3}{2}+\varepsilon]\\
0,& t\in [0,\overline{t}-(\varepsilon^\frac{3}{2}+\varepsilon)]\cup[\overline{t}+\varepsilon^\frac{3}{2}+\varepsilon, T/6]
\end{array}.\right.$$

Set $X_\varepsilon(t)=(x_\varepsilon(t),y_\varepsilon(t),z_\varepsilon(t))$ where:
$$x_\varepsilon(t)=x(t)+\varepsilon f_\varepsilon(t),\quad y_\varepsilon(t)=y(t), \quad z_\varepsilon(t)=z(t).$$
Notice that $X_\varepsilon\in\Omega_\frac{1}{6}$. We will prove that $\mathcal{A}_\frac{1}{6}(X_\varepsilon)<\mathcal{A}_\frac{1}{6}(X|_{[0,T/6]})$, which is a contradiction. This difference beteween the action of $X_\varepsilon$ and that that of $X|_{[0,T/6]}$ can be written as
$$\Delta\mathcal{A}_\frac{1}{6}(\varepsilon)=\mathcal{A}_\frac{1}{6}(X_\varepsilon)-\mathcal{A}_\frac{1}{6}(X|_{[0,T/6]})=\Delta\mathcal{A}_\frac{1}{6}^+(\varepsilon)+\Delta\mathcal{A}_\frac{1}{6}^-(\varepsilon),$$
where
$$\Delta\mathcal{A}_\frac{1}{6}^-(\varepsilon)=\int^{\overline{t}}_{\overline{t}-(\varepsilon^\frac{2}{3}+\varepsilon)}(L(X_\varepsilon,\dot{X}_\varepsilon)-L(X,\dot{X}))(t)dt,$$ 
$$\Delta\mathcal{A}_\frac{1}{6}^+(\varepsilon)=\int_{\overline{t}}^{\overline{t}+(\varepsilon^\frac{2}{3}+\varepsilon)}(L(X_\varepsilon,\dot{X}_\varepsilon)-L(X,\dot{X}))(t)dt,$$

As in proposition $\ref{total collision}$, we decompose $\Delta\mathcal{A}_\frac{1}{6}^+(\varepsilon)$ as the sum of three terms:
$$\Delta\mathcal{A}_\frac{1}{6}^+(\varepsilon)=A_1^+(\varepsilon)+A_2^+(\varepsilon)+A_3^+(\varepsilon),$$
$$A_1^+(\varepsilon)=\int_{\overline{t}}^{\overline{t}+\varepsilon^\frac{3}{2}}(U(X_\varepsilon)-U(X))(t)dt,$$
$$A_2^+(\varepsilon)=\int^{\overline{t}+\varepsilon^\frac{3}{2}+\varepsilon}_{\overline{t}+\varepsilon^\frac{3}{2}}(U(X_\varepsilon)-U(X))(t)dt,$$
$$A_3^+(\varepsilon)=\int^{\overline{t}+\varepsilon^\frac{3}{2}+\varepsilon}_{\overline{t}+\varepsilon^\frac{3}{2}}(K(X_\varepsilon,\dot{X}_\varepsilon)-K(X,\dot{X}))(t)dt.$$
To estimate $A_1^+$, observe that $A_1^+(\varepsilon)\leq \overline{A}_1^+(\varepsilon)$, where we set
$$\overline{A}_1^+(\varepsilon)=\frac{1}{8}\int_{\overline{t}}^{\overline{t}+\varepsilon^\frac{3}{2}}\left(\frac{1}{x(t)+\varepsilon}-\frac{1}{x(t)}\right)dt.$$
with $x(t)=x_0|t-\overline{t}|^\frac{2}{3}+O(|t-\overline{t}|).$

Changing the variables $t=\overline{t}+(\varepsilon\tau)^\frac{3}{2}$, and making the same computation as in proposition $\ref{total collision}$, we get the existence of a constant $k_+>0$, independent of $\varepsilon$, such that 
$$\overline{A}_1^+(\varepsilon)=-k_+\varepsilon^\frac{1}{2}+O(\varepsilon),$$
so $\overline{A}_1^+(\varepsilon)<0$ for $\varepsilon$ sufficiently small. Since $x(t)$ and $f_\varepsilon(t)$ are positive, we have
$$A_2^+(\varepsilon)=\frac{1}{8}\int^{\overline{t}+\varepsilon^\frac{3}{2}+\varepsilon}_{\overline{t}+\varepsilon^\frac{3}{2}}\left(\frac{1}{x(t)+\varepsilon f_\varepsilon(t)}-\frac{1}{x(t)}\right)dt+\int^{\overline{t}+\varepsilon^\frac{3}{2}+\varepsilon}_{\overline{t}+\varepsilon^\frac{3}{2}}\left(U_0(X_\varepsilon(t))-U_0(X(t))\right)dt.$$
Hence $$A_2^+\leq O(\varepsilon^2).$$
Concerning $A_3^+$ we have
$$A_3^+(\varepsilon)=\int^{\overline{t}+\varepsilon^\frac{3}{2}+\varepsilon}_{\overline{t}+\varepsilon^\frac{3}{2}}(||\dot{X}_\varepsilon||^2-||\dot{X}||^2)dt=\int^{\overline{t}+\varepsilon^\frac{3}{2}+\varepsilon}_{\overline{t}+\varepsilon^\frac{3}{2}}(-2\dot{x}(t)+1)dt\leq O(\varepsilon).$$
We conclude that $\Delta\mathcal{A}_\frac{1}{6}^+\leq-k_+\varepsilon^\frac{1}{2}+O(\varepsilon)$. Analogously, $\Delta\mathcal{A}_\frac{1}{6}^-\leq-k_-\varepsilon^\frac{1}{2}+O(\varepsilon)$, so $$\Delta\mathcal{A}_\frac{1}{6}\leq-(k_-+k_+)\varepsilon^\frac{1}{2}+O(\varepsilon),$$
and we have a contradiction.

Notice that if $\overline{t}=T/6$ we can define $f_\varepsilon(t)$ as
$$f_\varepsilon(t)=\left\{\begin{array}{cc}
0,& t\in [0,T/6-(\varepsilon^\frac{3}{2}+\varepsilon)]\\
\frac{\varepsilon+\varepsilon^\frac{3}{2}+(t-T/6)}{\varepsilon},& t\in[T/6-(\varepsilon^\frac{3}{2}+\varepsilon),T/6-\varepsilon^\frac{3}{2}]\\
1,& t\in[T/6-\varepsilon^\frac{3}{2},T/6]\\
\end{array},\right.$$
and repeat the argument above. We must compute only $\Delta\mathcal{A}_\frac{1}{6}^-(\varepsilon)$. This finishes our prove.
\end{proof}

In order to conclude our study of the collisions in the octahedral symmetric six-body problem we have to prove that there is no \textit{quadruple collision}, and that the double collisions are regularized. To prove the there is no quadruple collision we follow the method used to prove that there is no extra double collisions.

We argue by contradiction. First of all we remark that at $t=0$ we have $x(0)=0,y(0)=z(0)$, therefore we can have only a double collision or a total collision. We have already prove that a minimizer is free of total collisions. Hence at $t=0$ we have necessarily a double collision. The same argument holds for $t=T/3$ and $t=2T/3$. Assume now, by contradiction, that we have a quadruple collision at time $\overline{t}$. By symmetry we can assume that $\overline{t}\in(0,T/6]$. Without loss of generality, we will say that this quadruple collision is of the cluster $\{1,2,3,4\}$, that is to say: $x(\overline{t)}=y(\overline{t})$ and $z(\overline{t})>0$. The lagrangian and energy of the cluster are
$$L_{x,y}=\frac{1}{2}(\dot{x}^2+\dot{y}^2)+\frac{1}{8}\left(\frac{1}{x}+\frac{1}{y}\right)+\frac{1}{\sqrt{x^2+y^2}},$$
$$h_{x,y}=\frac{1}{2}(\dot{x}^2+\dot{y}^2)-\frac{1}{8}\left(\frac{1}{x}+\frac{1}{y}\right)-\frac{1}{\sqrt{x^2+y^2}}.$$
With this notations, we will prove the following proposition.

\begin{prop}\label{h_{x,y} abs cont}Let $X=(x,y,z)$ a minimizer of $\mathcal{A}|_\Omega$ and $\overline{t}\in(0,T/6]$ the instant of a quadruple collision of $m_1,m_2,m_3,m_4$. There is an interval, $[a,b]$, such that
\begin{equation}\label{flor} \overline{t}\in[a,b]\subset(0,T/6]
\end{equation} 
and $t\mapsto h_{x,y}=h_{x,y}(X(t),\dot{X}(t))$ is an absolutely continuous function in the interval $[a,b]$. 
\end{prop}

\begin{proof}
Since $X$ is a minimizer of $\mathcal{A}|_\Omega$, $X|_{[0,T/6]}$ is a minimizer of $\mathcal{A}|_{\Omega_{\frac{1}{6}}}=\mathcal{A}_\frac{1}{6}$. So the function $h_{x,y}(t)=h_{x,y}(X(t),\dot{X}(t))$ is well defined on a subset of $[0,T/6]$ with full measure as in $(\ref{flor})$, and such that only quadruple collisions of the masses on the $x,y$-axis occur at $t\in[a,b]$. Let $\delta$ be a function of class $\mathcal{C}^1$ with support contained in $[a,b]$. Let
$$t_\varepsilon:[0,T/6]\rightarrow[0,T/6], \quad\quad \tau\mapsto t_\varepsilon(\tau)=\tau+\varepsilon\delta(\tau).$$
As in proposition $\ref{extra double collision}$, $t_\varepsilon$ is monotonic and bijective if $\varepsilon>0$ is small, and the equations $(\ref{pre :)}),(\ref{:)})$ hold.

Let $X:[0,T/6]\rightarrow S, \quad X_\varepsilon(\tau)=(x(t_\varepsilon),y(t_\varepsilon),z(\tau))$ a variation of $X|_{[0,\frac{T}{6}]}$. Define
$$\begin{array}{c}
\mathcal{A}_\frac{1}{6}^{x,y}(X_\varepsilon)  =\int_0^\frac{T}{6}\left(K_{x,y}(X_\varepsilon(t_\varepsilon(\tau)))+U_{x,y}(X_\varepsilon(t_\varepsilon(\tau)))\right)d\tau\\
 =\int_0^\frac{T}{6}\left(\frac{1}{2}\left(\dot{x}(t_\varepsilon(\tau))^2+\dot{y}(t_\varepsilon(\tau))^2\right)+\frac{1}{8x(t_\varepsilon(\tau))}+\frac{1}{8y(t_\varepsilon(\tau))}+\frac{1}{\sqrt{x(t_\varepsilon(\tau))^2+y(t_\varepsilon(\tau))^2}}\right)d\tau;
\end{array}$$
and
$$\mathcal{A}_\frac{1}{6}^1(X_\varepsilon)=\int_0^\frac{T}{6}\left(K_1(X_\varepsilon(t_\varepsilon(\tau)))+U_1(X_\varepsilon(t_\varepsilon(\tau)))\right)d\tau,$$
where we set
$$K_1=\frac{1}{2}\dot{z}^2,$$
$$U_1=\left(\frac{1}{\sqrt{x^2+z^2}}+\frac{1}{\sqrt{y^2+z^2}}\right)+\frac{1}{8z}.$$
After making the change of variables $\tau=\tau_\varepsilon(t)$ in $\mathcal{A}_\frac{1}{6}^{x,y}(X_\varepsilon)$, we have
$$\mathcal{A}_\frac{1}{6}^{x,y}(X_\varepsilon)=\int_0^\frac{T}{6}\left[\frac{1}{2}\left(\dot{x}(t)^2+\dot{y}(t)^2\right)\frac{1}{\frac{d\tau_\varepsilon}{dt}}+\left(\frac{1}{8x(t)}+\frac{1}{8y(t)}+\frac{1}{\sqrt{x(t)^2+y(t)^2}}\right)\frac{d\tau_\varepsilon}{dt}\right]dt.$$
Since $\left.X\right|_{[0,T/6]}$ is a minimizer of $\left.\mathcal{A}_\frac{1}{6}\right|_\Omega$ and $\varepsilon\rightarrow \mathcal{A}_\frac{1}{6}(X_\varepsilon)$ is a differentiable function, using $(\ref{:)})$ we obtain
$$\left.\frac{d}{d\varepsilon}\mathcal{A}_\frac{1}{6}(X_\varepsilon)\right|_{\varepsilon=0}=\int_a^b\left(h_{x,y}\dot{\delta}+\left(\frac{\partial U_1}{\partial x}\dot{x}+\frac{\partial U_1}{\partial y}\dot{y}\right)\delta\right)dt=0\quad\quad \forall \delta.$$
As in proposition $\ref{extra double collision}$, we use the Fundamental Lemma of the Calculus of Variations,  in order to conclude that $t\mapsto h_{x,y}(t)$ is absolutely continuous and also
$$\dot{h}_{x,y}(t)=\nabla_{x,y}U_1(X(t))\cdot(\dot{x},\dot{y}),$$
where we set $$\nabla_{x,y}U_1(X(t))=\left(\frac{\partial U_1}{\partial x},\frac{\partial U_1}{\partial y}\right).$$
\end{proof}

\begin{cor}Let $X=(x,y,z)\in\Omega$ be a minimizer and $\overline{t}\in[0,T/6]$ an instant of quadruple collision, say $x(\overline{t})=y(\overline{t})=0$. Then $\overline{t}$ is isolated in the set of collision times, $T_c(X)$.
\end{cor}

\begin{rem} As we already proved in theorem $\ref{thm extra double collision}$, double collisions can not accumulate in a quadruple colllision and there is no extra double collision! We have proved also that there is no total colission. So, we can choose a neighborhood of $\overline{t}$ such that there is no double collision such that $x=0$ or $y=0$. 
\end{rem}

\begin{proof}
Let $X$ be a minimizer of $\left.\mathcal{A}\right|_\Omega$, $\overline{t}\in(0,T/6)$ a quadruple collision time at which $x=y=0$ and $T_c(X)=\{t^*: x(t^*)y(t^*)z(t^*)=0\}$. By contradiction, suppose that $\overline{t}$ is not isolated in $T_c(X)$. Notice, by the remark above, that there is no possible acumulations of extra double collisions or total collisions to a quadruple collision. So, only collisions of $m_1,m_2,m_3,m_4$ can accumulate in $\overline{t}$,  

we can choose an interval $[a,b]$ such that $\overline{t}\in(a,b)$ and for all $t\in[a,b]$ there is no collision of $m_1$, $m_2$, $m_3$ or $m_4$ with the others masses, $m_5$, $m_6$. We know  that $T_c$ is closed and has zero measure. So there exists a sequence of compact intervals $([a_n,b_n])_{n\in\mathbb{N}}$ such that $a_n,b_n\rightarrow \overline{t}$ and
$$x(a_n)=x(b_n)=y(a_n)=y(b_n)=0 \quad\quad \text{and}\quad\quad I_{x,y}(t)\neq 0\quad\forall t\in(a_n,b_n),$$
where $I_{x,y}(t)$ is the moment of inertia of the cluster, which can be written as $I_{x,y}(t)=I_{x,y}(X(t))=x(t)^2+y(t)^2$. Therefore, $I_{x,y}(a_n)=I_{x,y}(b_n)=0$ and $I_{x,y}(t)\neq0$ $
\forall t\in(a_n,b_n)$. By proposition $\ref{without collisions}$, $X$ is a solution of the octahedral symmetric six-body problem in $(a_n,b_n)$ and $I_{x,y}(X(t))$ is $\mathcal{C}^2$ in $(a_n,b_n)$. If $s_n\in(a_n,b_n)$ denotes the maximum of $I_{x,y}$ in $[a_n,b_n]$, we have $\ddot{I}_{x,y}(s_n)\leq0$. The equations of motion tells us that

$$(\ddot{x},\ddot{y})=\nabla_{x,y}U_{x,y}+\nabla_{x,y}U_1.$$
And with this we have the generalized Lagrange-Jacobi relation

\begin{align}
\ddot{I}_{x,y}(t) & =4K_{x,y}+2(x,y)\cdot\left(\nabla_{x,y}U_{x,y}+\nabla_{x,y}U_1\right)\nonumber\\
 & =4h_{x,y}(t)+2U_{x,y}+2(x,y)\cdot\nabla_{x,y}U_1.\nonumber
\end{align}
But, according to the proposition $\ref{h_{x,y} abs cont}$, $h_{x,y}$ is absolutely continuous function on $[a,b]$, so $h_{x,y}$ and $(x,y)\cdot\nabla_{x,y}U_1$ are bounded and $U_{x,y}(t)\stackrel{t\rightarrow\overline{t}}{\longrightarrow}+\infty$. Then $I_{x,y}(s_n)\stackrel{n\rightarrow+\infty}{\longrightarrow}+\infty$, ta contradiction. The corollary is proved.
\end{proof}

\begin{thm} A minimizer of $\mathcal{A}|_\Omega$ does not have a quadruple collision.
\end{thm}

\begin{proof}
We will make a small deformation of the hypothetical minimizer with a collision which decreases the action, similar to what we did in propositions $\ref{total collision}$ and $\ref{thm extra double collision}$. Notice that by the symmetry of the problem a quadruple collision in $\{t\in\mathbb{R}:\quad t\in \frac{T}{3}\mathbb{Z}\}$ means a total collision. From proposition \ref{total collision} it can not happen. Moreover, by theorem \ref{thm extra double collision} there is no collision in $(0,T/3)$. Suppose that there is a minimizer of $\left.\mathcal{A}\right|_\Omega$ with an extra quadruple collision at the instant $\overline{t}\in(0,T/6)$ among the masses on the $x,y$-axis. By Sundman's estimates we have
$$\left\{\begin{array}{cccc}
x(t)=x_0|t-\overline{t}|^\frac{2}{3}+O(t-\overline{t}), & \text{and} & \dot{x}(t)=\frac{2}{3}x_0|t-\overline{t}|^{-\frac{1}{2}}+O(1)\\
y(t)=y_0|t-\overline{t}|^\frac{2}{3}+O(t-\overline{t}), & \text{and} & \dot{y}(t)=\frac{2}{3}y_0|t-\overline{t}|^{-\frac{1}{2}}+O(1)
\end{array},
\right.
$$
as $t\rightarrow\overline{t}$, where $x_0,y_0>0$.  
For $\varepsilon>0$, define $f_\varepsilon:[0,T/6]\rightarrow\real_+$, where
$$f_\varepsilon(t)=\left\{\begin{array}{cc}
\frac{\varepsilon+\varepsilon^\frac{3}{2}+(t-\overline{t})}{\varepsilon},& t\in[\overline{t}-(\varepsilon^\frac{3}{2}+\varepsilon),\overline{t}-\varepsilon^\frac{3}{2}]\\
1,& t\in[\overline{t}-\varepsilon^\frac{3}{2},\overline{t}+\varepsilon^\frac{3}{2}]\\
\frac{\varepsilon+\varepsilon^\frac{3}{2}+(\overline{t}-t)}{\varepsilon},& t\in[\overline{t}+\varepsilon^\frac{3}{2},\overline{t}+\varepsilon^\frac{3}{2}+\varepsilon]\\
0,& t\in [0,\overline{t}-(\varepsilon^\frac{3}{2}+\varepsilon)]\cup[\overline{t}+\varepsilon^\frac{3}{2}+\varepsilon, T/6]
\end{array}.\right.$$
Set $X_\varepsilon(t)=(x_\varepsilon(t),y_\varepsilon(t),z_\varepsilon(t))$ where:
$$x_\varepsilon(t)=x(t)+\varepsilon f_\varepsilon(t),\quad y_\varepsilon(t)=y(t) +\varepsilon f_\varepsilon(t), \quad z_\varepsilon(t)=z(t),$$
Notice that $X_\varepsilon\in\Omega_\frac{1}{6}$. We will prove that $\mathcal{A}_\frac{1}{6}(X_\varepsilon)<\mathcal{A}_\frac{1}{6}(X|_{[0,T/6]})$, and it is a contradiction. This difference can be written as
$$\Delta\mathcal{A}_\frac{1}{6}(\varepsilon)=\mathcal{A}_\frac{1}{6}(X_\varepsilon)-\mathcal{A}_\frac{1}{6}(X|_{[0,T/6]})=\Delta\mathcal{A}_\frac{1}{6}^+(\varepsilon)+\Delta\mathcal{A}_\frac{1}{6}^-(\varepsilon),$$
where $\Delta\mathcal{A}_\frac{1}{6}^+(\varepsilon)$ and $\Delta\mathcal{A}_\frac{1}{6}^-(\varepsilon)$ are defined like in theorem $\ref{thm extra double collision}$. Again

\begin{align}
\Delta\mathcal{A}_\frac{1}{6}^+(\varepsilon)=&A_1^+(\varepsilon)+A_2^+(\varepsilon)+A_3^+(\varepsilon),\nonumber\\
A_1^+(\varepsilon)=&\int_{\overline{t}}^{\overline{t}+\varepsilon^\frac{3}{2}}(U(X_\varepsilon)-U(X))(t)dt,\nonumber\\
A_2^+(\varepsilon)=&\int^{\overline{t}+\varepsilon^\frac{3}{2}+\varepsilon}_{\overline{t}+\varepsilon^\frac{3}{2}}(U(X_\varepsilon)-U(X))(t)dt,\nonumber\\
A_3^+(\varepsilon)=&\int^{\overline{t}+\varepsilon^\frac{3}{2}+\varepsilon}_{\overline{t}+\varepsilon^\frac{3}{2}}(K(X_\varepsilon,\dot{X}_\varepsilon)-K(X,\dot{X}))(t)dt.\nonumber
\end{align}

To estimate $A_1^+$, observe that $A_1^+(\varepsilon)\leq \overline{A}_1^+(\varepsilon)$, since $f_\varepsilon(t)\geq0$, where
$$\overline{A}_1^+(\varepsilon)=\frac{1}{8}\int_{\overline{t}}^{\overline{t}+\varepsilon^\frac{3}{2}}\left(\frac{1}{x_0|t-\overline{t}|^\frac{2}{3}+\varepsilon+O(|t-\overline{t}|)}-\frac{1}{x_0|t-\overline{t}|^\frac{2}{3}+O(|t-\overline{t}|)}\right)dt.$$
We have proved in theorem $\ref{thm extra double collision}$ that $\overline{A}_1^+(\varepsilon)=-k_+\varepsilon^\frac{1}{2}+O(\varepsilon)$, where $k_+$ is a positive constant. Since $x(t)$, $y(t)$ and $f_\varepsilon(t)$ are positive, we have

$$
A_2^+(\varepsilon)\leq\frac{1}{8}\int^{\overline{t}+\varepsilon^\frac{3}{2}+\varepsilon}_{\overline{t}+\varepsilon^\frac{3}{2}}\left(\frac{1}{x(t)+\varepsilon f_\varepsilon(t)}+\frac{1}{y(t)+\varepsilon f_\varepsilon(t)}-\left(\frac{1}{x(t)}+\frac{1}{y(t)}\right)\right)dt$$
$$ +\int^{\overline{t}+\varepsilon^\frac{3}{2}+\varepsilon}_{\overline{t}+\varepsilon^\frac{3}{2}}\left(U_1(X_\varepsilon(t))-U_1(X(t))\right)dt,$$
$$A_2^+\leq O(\varepsilon).$$
For $A_3^+$ we have
$$A_3^+(\varepsilon)=\int^{\overline{t}+\varepsilon^\frac{3}{2}+\varepsilon}_{\overline{t}+\varepsilon^\frac{3}{2}}(||\dot{X}_\varepsilon||^2-||\dot{X}||^2)dt=\int^{\overline{t}+\varepsilon^\frac{3}{2}+\varepsilon}_{\overline{t}+\varepsilon^\frac{3}{2}}(-2\dot{x}(t)+1)dt\leq O(\varepsilon).$$

We conclude that $\Delta\mathcal{A}_\frac{1}{6}^+\leq-k_+\varepsilon^\frac{1}{2}+O(\varepsilon)$. Analogously, $\Delta\mathcal{A}_\frac{1}{6}^-\leq-k_-\varepsilon^\frac{1}{2}+O(\varepsilon)$, so $$\Delta\mathcal{A}_\frac{1}{6}\leq-(k_-+k_+)\varepsilon^\frac{1}{2}+O(\varepsilon),$$
and we have a contradiction.

Notice that if $\overline{t}=T/6$ we can define $f_\varepsilon(t)$ by
$$f_\varepsilon(t)=\left\{\begin{array}{cc}
0,& t\in [0,T/6-(\varepsilon^\frac{3}{2}+\varepsilon)]\\
\frac{\varepsilon+\varepsilon^\frac{3}{2}+(t-T/6)}{\varepsilon},& t\in[T/6-(\varepsilon^\frac{3}{2}+\varepsilon),T/6-\varepsilon^\frac{3}{2}]\\
1,& t\in[T/6-\varepsilon^\frac{3}{2},T/6]\\
\end{array},\right.$$
and repeat the argument above. We need to compute only $\Delta\mathcal{A}_\frac{1}{6}^-(\varepsilon)$. That finishes our prove.

\end{proof}

\section{Regularization and proof of the main theorem}

In the last section we proved that the collisions are only double collisions occurring in $\{t\in\mathbb{R}:\quad t=T/3\mathbb{Z}\}$. We will consider now $\dot{X}=(\dot{x},\dot{y},\dot{z})$ as independent variables in the phase space. So the hamiltonian 
$$H(X,\dot{X})=\frac{1}{2}(\dot{x}^2+\dot{y}^2+\dot{z}^2)-\left(\frac{1}{\sqrt{x^2+y^2}}+\frac{1}{\sqrt{x^2+z^2}}+\frac{1}{\sqrt{y^2+z^2}}\right)-\frac{1}{8}\left(\frac{1}{x}+\frac{1}{y}+\frac{1}{z}\right),$$
is singular only at double collisions. Now we will \textit{regularize} the equations of motions at the double collisions by a time scaling, using a kind of Levi-Civita regularization ~\cite{Levi-Civita}. Introduce the new variables $(\gamma,\upsilon,\zeta,\Gamma,\Upsilon, Z)$ defined by 
\begin{equation}\label{variables}
\left\{\begin{array}{cc}
x=\gamma^2, & \Gamma=2\dot{x}\gamma\\
y=\upsilon^2, & \Upsilon=2\dot{y}\upsilon\\
z=\zeta^2, & Z=2\dot{z}\zeta
\end{array}\right..
\end{equation} 
Notice that $(\theta,\Theta)=(\gamma,\upsilon,\zeta,\Gamma,\Upsilon,Z)\mapsto (X,\dot{X})=(x,y,z,\dot{x},\dot{y},\dot{z})$ is a symplectic $8$-fold covering outside collisions. Equations of motions in the new variables are still in hamiltonian form. The hamiltonian can be written as
$$H(\theta,\Theta)=\frac{1}{2}\left(\frac{\Gamma^2}{4\gamma^2}+\frac{\Upsilon^2}{4\upsilon^2}+\frac{Z^2}{4\zeta^2}\right)-\left(\frac{1}{\sqrt{\gamma^4+\upsilon^4}}+\frac{1}{\sqrt{\gamma^4+\zeta^4}}+\frac{1}{\sqrt{\upsilon^4+\zeta^4}}\right)-$$ $$-\frac{1}{8}\left(\frac{1}{\gamma^2}+\frac{1}{\upsilon^2}+\frac{1}{\zeta^2}\right),$$
and the equations of motion are 

\begin{equation}\label{eq antes reg}
\left\{\begin{array}{cccc}
\dot{\gamma}= & \frac{\Gamma}{4\gamma^2} & & \\
\dot{\Gamma}=& \left(\frac{\Gamma^2}{4\gamma^3} -\frac{1}{4\gamma^3}\right)-& \frac{2\gamma^3}{(\gamma^4+\upsilon^4)^\frac{3}{2}}&-\frac{2\gamma^3}{(\gamma^4+\zeta^4)^\frac{3}{2}}\\
\dot{\upsilon}= & \frac{\Upsilon}{4\upsilon^2} & & \\
\dot{\Upsilon}=& \left(\frac{\Upsilon^2}{4\upsilon^3} -\frac{1}{4\upsilon^3}\right)-& \frac{2\upsilon^3}{(\upsilon^4+\gamma^4)^\frac{3}{2}}&-\frac{2\upsilon^3}{(\upsilon^4+\zeta^4)^\frac{3}{2}}\\
\dot{\zeta}= & \frac{Z}{4\zeta^2} & & \\
\dot{Z}=& \left(\frac{Z^2}{4\zeta^3} -\frac{1}{4\zeta^3}\right)-& \frac{2\zeta^3}{(\zeta^4+\gamma^4)^\frac{3}{2}}&-\frac{2\zeta^3}{(\upsilon^4+\zeta^4)^\frac{3}{2}}
\end{array}\right..
\end{equation}

In order to regularize, we suppose that the total energy is constant, \textit{i.e.} $H(\theta,\Theta)=h$. So

\begin{align}
\left(\frac{\Gamma^2}{8\gamma^2}-\frac{1}{8\gamma^2}\right)=&h-\frac{1}{2}\left(\frac{\Upsilon^2}{4\upsilon^2}+\frac{Z^2}{4\zeta^2}\right)+\left(\frac{1}{\sqrt{\gamma^4+\upsilon^4}}+\frac{1}{\sqrt{\gamma^4+\zeta^4}}+\frac{1}{\sqrt{\upsilon^4+\zeta^4}}\right)\nonumber\\
&+\frac{1}{8}\left(\frac{1}{\upsilon^2}+\frac{1}{\zeta^2}\right)\nonumber
\end{align}

Define now 
$$\begin{aligned}
f(\gamma,\upsilon,\zeta,\Gamma,\Upsilon,Z)=&\frac{1}{2}\left(\frac{\Upsilon^2}{4\upsilon^2}+\frac{Z^2}{4\zeta^2}\right)-\left(\frac{1}{\sqrt{\gamma^4+\upsilon^4}}+\frac{1}{\sqrt{\gamma^4+\zeta^4}}+\frac{1}{\sqrt{\upsilon^4+\zeta^4}}\right)\\
&-\frac{1}{8}\left(\frac{1}{\upsilon^2}+\frac{1}{\zeta^2}\right).
\end{aligned} $$
Let us introduce the time scaling $$'=(\gamma^2\upsilon^2\zeta^2)\quad\dot{}\quad.$$
This time scaling will permit us to regularize all the doubles collisions at once.\\
The system $(\ref{eq antes reg})$ in the new time paramenter is
\begin{equation}\label{eq depois reg}
\left\{\begin{array}{cccc}
\gamma'= & \frac{1}{4}\upsilon^2\zeta^2\Gamma & & \\
\Gamma'=& 2\gamma\upsilon^2\zeta^2(h-f(\gamma,\upsilon,\zeta,\Gamma,\Upsilon,Z))-& \frac{2\gamma^5\upsilon^2\zeta^2}{(\gamma^4+\upsilon^4)^\frac{3}{2}}&-\frac{2\gamma^5\upsilon^2\zeta^2}{(\gamma^4+\zeta^4)^\frac{3}{2}}\\
\upsilon'= & \frac{1}{4}\gamma^2\zeta^2\Upsilon & & \\
\Upsilon'=& 2\upsilon\gamma^2\zeta^2(h-f(\upsilon,\gamma,\zeta,\Upsilon,\Gamma,Z))-& \frac{2\gamma^2\upsilon^5\zeta^2}{(\gamma^4+\upsilon^4)^\frac{3}{2}}&-\frac{2\gamma^2\upsilon^5\zeta^2}{(\upsilon^4+\zeta^4)^\frac{3}{2}}\\
\zeta'= & \frac{1}{4}\gamma^2\upsilon^2Z & & \\
Z'=& 2\zeta\gamma^2\upsilon^2(h-f(\zeta,\upsilon,\gamma,Z,\upsilon,\Gamma))-& \frac{2\gamma^2\upsilon^2\zeta^5}{(\gamma^4+\zeta^4)^\frac{3}{2}}&-\frac{2\gamma^2\upsilon^2\zeta^5}{(\upsilon^4+\zeta^4)^\frac{3}{2}}\\
\end{array}\right..
\end{equation}

\begin{defn} For a minimizer of $\mathcal{A}|_\Omega$,  $X$, a double collision at the instant $t=\overline{t}$ is regularized if and only if:
\begin{itemize}
\item [(i)] the total energy is constant;
\item [(ii)] For $t$ in a neighborhood of $\overline{t}$, the lift of $t\mapsto(X,\dot{X})(t)$ by the sympletic covering (\ref{variables}) is a solution of the equations (\ref{eq depois reg}) after time scaling.
\end{itemize}
\end{defn}
\begin{prop} \label{regularizacao} Double collisions of minimizers of $\mathcal{A}|_\Omega$ in the octahedral symmetric six-body problem are regularized.
\end{prop}
\begin{proof}
By the symmetry of the problem and the invariance of $X$ under the action of $G$, we need to check only the collision at $\overline{t}=0$. Notice that our rescaling of time $\quad'=\gamma^2\upsilon^2\zeta^2\quad\dot{}\quad$ regularizes simultaneously all double collisions. 

We know that for a minimizer of $\mathcal{A}|_\Omega$ the total energy $h(t)=H(X(t),\dot{X}(t))$ is constant in each interval without collision, but in a small interval containing $\overline{t}=0$, a priori, the energy could be different for $t>0$ and $t<0$. We remark that
$$h(t)=h_x(t)+K_0-U_0(t),$$
where $K_0(X(t),\dot{X}(t))$ and $U_0(t)$ are the continuous function in a neighborhood of $\overline{t}$ defined in page \pageref{U_0 and K_0}. By proposition $\ref{extra double collision}$, $h_x$ is an absolutely continuous function in a neighborhood of a double collision on the $x$ axis. So, $t\mapsto h(t)$ is continuous in a neighborhood of $\overline{t}=0$. Therefore, is constant in $(-T/3,T/3)$, consequently in $[0,T]$.

The change of variables $\quad'=\gamma^2\upsilon^2\zeta^2\quad\dot{}\quad$ can be written as $\frac{d}{ds}=\gamma^2\upsilon^2\zeta^2\frac{d}{dt}$, where $\dot{}$ is the derivate with respect to the physical time, $t$, and $'$ is the derivate with respect to the \textit{new} (\textit{i.e.} rescaled) time $s$. The new time can be written as a function of the physical time $t$ as
$$s(t)=\int_0^t\frac{1}{x(u)y(u)z(u)}du,$$
and denote by $t=t(s)$ the inverse of $s=s(t)$. As we are studying a neighborhood that contains only a collision in the $x-$axis, we can use Sundman estimates and write:
$$\left\{\begin{array}{ccc}
x(t)=& x_0t^\frac{2}{3}&+O(t)\\
y(t)=& y_0 & +O(1)\\
z(t)=& z_0 & +O(1)
\end{array}
\right.\Longrightarrow x(t)y(t)z(t)=at^\frac{2}{3}+O(t),
$$
where $a=x_0y_0z_0>0$ is constant. Hence the integral above converges.

Let us term
$$\left\{\begin{array}{cc}
\gamma(s)=\sqrt{x(t(s))},& \Gamma(s)=2\dot{x}(t(s))\gamma(s)\\
\upsilon(s)=\sqrt{y(t(s))},& \Upsilon(s)=2\dot{y}(t(s))\upsilon(s)\\
\zeta(s)=\sqrt{z(t(s))},& Z(s)=2\dot{z}(t(s))\zeta(s)\\
\end{array}
\right.,$$
so the path
\begin{equation}\label{path}
s\longmapsto(\gamma(s),\Gamma(s),\upsilon(s),\Upsilon(s),\zeta(s),Z(s))
\end{equation}
is a solution of $(\ref{eq depois reg})$ in $(0,\epsilon)$, for $\epsilon$ suficiently small. Since the system $(\ref{eq depois reg})$ is regular at double collisions in each energy surface, the path $(\ref{path})$ can be continued to $s\leq0$. By the symmetry of the solution, $\dot{z}(0)=-\dot{y}(0)$ and $z(0)=y(0)$, hence we choose an appropriate lift: $\upsilon(0)=\zeta(0)$ and $\Upsilon(0)=-Z(0)$. Exploiting the symmetry of differential equations (\ref{eq depois reg}), we get
$$\begin{array}{cc}
\gamma(s)=-\gamma(-s), & \Gamma(s)=\Gamma(-s)\\
\upsilon(s)=\zeta(-s), & \Upsilon(s)=-Z(-s)\\
\zeta(s)=\upsilon(-s), & Z(s)=-\Upsilon(-s)
\end{array} .
$$
In physical coodinates $x,y,z$ and in physical time $t$ we get 
$$x(-t)=x(t)\quad,\quad y(-t)=z(t)\quad,\quad z(-t)=y(t).$$
This proves that for $t>0$, $X(t)$ coincides with the solution obtained through Levi-Civita regularization. For the others double collisions the argument is the same.
\end{proof}

Let $X=(x,y,z)$ be a minimizer of $\mathcal{A}|_\Omega$. By proposition $\ref{regularizacao}$ $X$ has double collisions regularized, and by proposition $\ref{without collisions}$, $X$ is a solution of the $6$-body problem. So we can say that $X$ is a \textit{collision solution} of the symmetric $6$-body problem on the axis. To prove the Main Theorem, we just have to prove $(ii)$, because $(i)$, $(iii)$ and $(iv)$ are already satisfied. Notice that, by the symmetry of the problem, we need to prove that $t\mapsto y(t)$ is decreasing and $t\mapsto z(t)$ is increasing in $[0, T/6]$ is the same thing as prove $t\mapsto x(t)$ is decreasing in $[2T/3,5T/6]$ and increasing in $[T/3,T/2]$, respectively. And, prove that $t\mapsto y(t)$ is decreasing and $t\mapsto z(t)$ is decreasing in $[T/6,T/3]$ is the same thing as proving that $t\mapsto x(t)$ is decreasing in $[5T/6,T]$ and decreasing in $[T/2,2T/3]$, respectively. So, we have just study the behavior of $t\mapsto x(t)$ in a period. By the equations of motion ($\ref{eq movi q}$), we have
\begin{equation}
\ddot{x}(t)=\frac{\partial U}{\partial x}<0 \quad \forall t\in[0,T].
\end{equation}
Therefore $x(t)=x(-t)$, so $\dot{x}(T/2)=0$. We conclude that $t\mapsto x(t)$ is increasing in $[0,T/2]$ and decreasing in $[T/2,T]$. And it finish the prove of the Main Theorem.

\vspace{1cm}

\textbf{Acknowledgment} I want to thanks professor Andrea Venturelli who proposed this problem and  strongly encouraged me to write this paper and  professor Eduardo Leandro for all his support. I wish to thank CAPES, for the pos-doc schollarship. \\

\textbf{Conflict of Interest:} The author has received research grants from CAPES (Coordena\c c\~ao de Aperfei\c coamento de Pessoal de N\'ivel Superior).


\begin{thebibliography}{99}

\bibitem{Brezis} Brezis, H. , \textit{Analyse fonctionnelle, Théorie et applications}, Dunod, Paris, 1999.
\bibitem{Chang} Chen, K-C., \textit{Variational Methods and Periodic Solutions of Newtonian N-Body Problems}, Ph.D thesis, University of Minnesota, 2001.
\bibitem{CM} Chenciner, A., Montgomery, R., \textit{A Remarkable Periodic Solution of the Three-Body Problem in the Case of Equal Mass}, Annals of Mathematics 152, 2000, p. 881-901.
\bibitem{Coti Zelati} Coti-Zelati, Z., \textit{Periodic Solution for $N$-Body Problems}, Ann. Inst. Henri Poincaré, Anal Non Linéaire 7, number 5, 1990,  p. 477-492.
\bibitem{DGM} Degiovanni, M., Gianonni, F. and Marino, A., \textit{Periodic Solutions of Dynamical Systems
with Newtonian Type Potentials}, Ann. Scuola Norm. Sup. Pisa Cl. Sci. 15, 1988, p. 467-494.
\bibitem{Gordon} Gordon, W. B., \textit{A Minimizing Property of Kleperian Orbits}, Amarican Journal of Math, 99 (1997), 961-971. 
\bibitem{Levi-Civita} Levi-Civita, T. \textit{Sur la Régularization du problème des trois corps}, Acta. Math, 42 (1920), 99-144. 
\bibitem{Palais} Palais, R. , \textit{The principle of symmetric criticality}, Comm. Math. Phys. 69 (1979), 19-30.
\bibitem{Poincare} Poincaré, H. \textit{Sur Les Solutions Périodiques et le Principe de Moindre Action}, C.R.A.S., 1896.
\bibitem{SeT} Serra, E. and Terracini, S. \textit{Collisionless Periodic Solutions to Some Three-Body Problems},
Arch. Rational Mech. Anal. 120, 1992, p. 305-325.
\bibitem{Shibayama} Shibayama, M., \textit{Minimizing Periodic Orbits with Regularizable Collisions in the $n$-
Body Problem}, Arch. Rational Mech. Anal. 199, 2011, p. 821-841.
\bibitem{Schubart} Schubart, J., \textit{Numerische Aufsuchung periodischer L\"{o}sungen im Dreik\"{o}rperproblem}, Astronom. Nachr. 283 (1956),  17-22.
\bibitem{Venturelli} Venturelli, A. \textit{A Variational proof of the existence of Von Schubart's Orbits}, Discrete and Continuous Dynamical Systems B 10, 2008, p. 699-717.
\bibitem{Venturelli these} Venturelli, A. \textit{Application de la minimisation de l'action au Problème des N corps dans le plan et dans l'espace}, Ph.D thesis, Université Denis Diderot in Paris, 2002.
\bibitem{Young} Young, L.C., \textit{Lectures on the Calculus of Variations and Optimal Control Theory}, W.B. Saunders Company, 1969.

\end{thebibliography}
\end{document}